\documentclass[12pt]{article}
\usepackage[english]{babel}
\usepackage{amsthm,amsfonts, amsbsy, amssymb,amsmath,graphicx}
\usepackage{mathrsfs}
\usepackage[all]{xy}

\newtheorem{theorem}{Theorem}

\newtheorem{corollary}{Corollary}

\newtheorem{proposition}{Proposition}
\theoremstyle{definition}
\newtheorem{definition}{Definition}
\theoremstyle{remark}
\newtheorem{remark}{Remark}
\newtheorem{example}{Example}

\newcommand{\Z}{{\mathbb Z}}

\newcommand{\R}{{\mathbb R}}

\newcommand{\V}{{\mathcal V}}

\newcommand{\K}{{\mathcal K}}

\date{}

\title{Intersection formulas for parities on virtual knots}
\author{Igor Nikonov}

\begin{document}

\maketitle

\begin{abstract}
We prove that parities on virtual knots come from invariant 1-cycles on the arcs of knot diagrams. In turn, the invariant cycles are determined by quasi-indices on the crossings of the diagrams. The found connection between the parities and the (quasi)-indices allows to define a new series of parities on virtual knots.
\end{abstract}

\section*{Introduction}

V.O. Manturov~\cite{M3} defined a parity as a rule to assign labels $0$ and $1$
to the crossings of knot diagrams in a way compatible with
Reidemeister moves. Capability to discriminate between odd and even
crossings allows to treat the crossings differently when
transforming knot diagrams or calculating their invariants.
The notion of parity allows to strengthen knot invariants, to prove minimality
theorem and to construct (counter)examples~\cite{IM, IMN1, M1,M2,M3,M4,M5,M6,M7}.

In~\cite{IMN} parities with coefficients in an arbitrary abelian group were defined, and in~\cite{N} oriented parity was introduced.

A well known parity-like invariant is the index polynomial. Starting Turaev's polynomial $u$~\cite{Turaev}, the index polynomial has appeared numerously in papers of A.~Henrich~\cite{Henrich}, Y.H.~Im, K.~Lee, S.Y.~Lee~\cite{ILL}, Z.~Cheng~\cite{Cheng}, H.~Dye~\cite{Dye}, L.H.~Kauffman~\cite{K2}, L.C.~Folwaczny~\cite{FK}, J.~Kim~\cite{Kim}, M.J.~Jeong~\cite{J}, H.~Gao, M.~Xu~\cite{GX}, N.~Petit~\cite{Petit}, K.~Kaur, M.~Prabhakar, A.~Vesnin~\cite{KPV} and others. The first axiomatization of index properties is due to Z. Cheng~\cite{Cheng2}. M.~Xu~\cite{Xu} found the general conditions for an index to define an invariant polynomial and defined~\emph{weak chord index}. See also~\cite{Cheng2,CFGMX}  for other index axiomatics and examples of chord indices.

The present paper establishes a new relation between the (oriented) parity and the (weak chord) index.

The content of the paper can be described by the following diagram.

\[
\xymatrix{
& & \txt{indices} \ar@{^{(}->}[d] \ar[dl]^-{\S\ref{sect:index_to_parity}}\\
\txt{oriented parities}\ar@/^1.5pc/[urr]^-{\S\ref{sect:derived_parity}} \ar@<0.3ex>[r]^-{\S\ref{subsect:parity_cycle}} & \txt{normalized invariant 1-cycles}\ar@<0.3ex>[l]\ar@<0.3ex>[r]^-{\S\ref{sect:quasiindex}} &\txt{quasi-indices}\ar@.@<0.6ex>[l]\\
& \txt{biquandle 1-cocycles}\ar[u]^-{\S\ref{subsect:cycle_biquandle}} &
}
\]

Using the technics of~\cite{N2}, in Section~\ref{sect:potentials} we define the parity cycle of an oriented parity. In fact, the parity cycle is some labeling of arcs of knot diagrams compatible with Reidemeister moves. The intersection formula~\eqref{eq:intersection_formula} establishes a bijection between parities and invariant cycles. In Section~\ref{subsect:cycle_biquandle} we show how parity cycles can be constructed from a biquandle 1-cocycle (up to some issue with colouring monodromy). In Section~\ref{sect:quasiindex} we return the labeling from arcs to crossings and prove that the parity cycle is determined by some quasi-index on the crossings of knot diagrams. Quasi-index generalizes the notion of index on crossings, and its nature is not clear yet. Section~\ref{sect:index_to_parity} is devoted to various formulas which restore a parity for a given (weak chord) index. Sections~\ref{sect:long_knots},\ref{sect:parity_functors} and~\ref{sect:links} describe how the intersection formula can be adapted to long virtual knots, parity functors and virtual links. In the last section~\ref{sect:derived_parity} we loop the intersection formula by considering parities as (signed) indices, and define derived parities.

\section{Definitions}

\subsection{Virtual knots}

We start with the definitions of virtual knots.

A \emph{$4$-graph} is any union of four-valent graphs and \emph{trivial components}, i.e. circles  considered as graphs without vertices and with one (closed) edge.
A \emph{virtual diagram} is an embedding of a $4$-graph into plane so that each vertex of the graph is marked as either \emph{classical} of \emph{virtual} vertex.
At a classical vertex one a pair of opposite edges (called \emph{overcrossing}) is chosen. The other pair of opposite edges at the vertex is called \emph{undercrossing}.

Virtual crossings of a virtual diagram  are usually drawn circled. The undercrossing of a classical vertex is drawn with a broken line whereas the overcrossing is drawn with a solid line (see Fig.~\ref{fig:virtual_trefoil}).
\begin{figure}[h]
\centering
\includegraphics[width=0.15\textwidth]{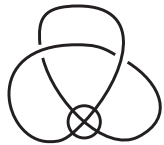}
\caption{A virtual trefoil diagram with two classical and one virtual crossings}\label{fig:virtual_trefoil}
\end{figure}
%

\emph{Moves of virtual diagrams} include classical Reidemeister moves ($R1$, $R2$, $R3$) and \emph{detour moves ($DM$)} that replace any diagram arc, which has only virtual crossings, with a new arc, which has the same ends and contains only virtual crossings (see Fig.~\ref{pic:virt_moves}). An equivalence class of virtual diagram modulo moves is called a \emph{virtual link}~\cite{K1}.

\begin {figure}[h]
\centering
\includegraphics[width=0.3\textwidth]{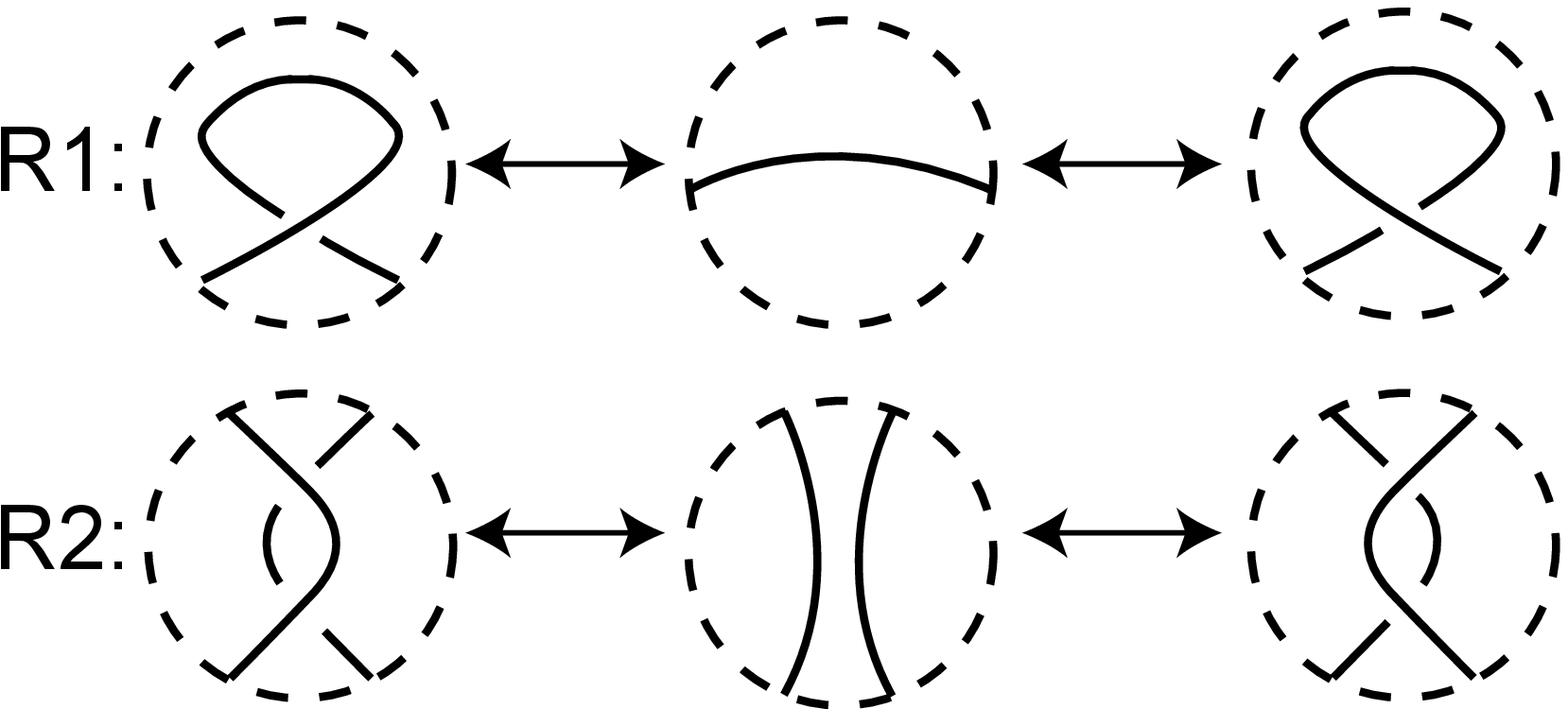}\qquad
\includegraphics[width=0.3\textwidth]{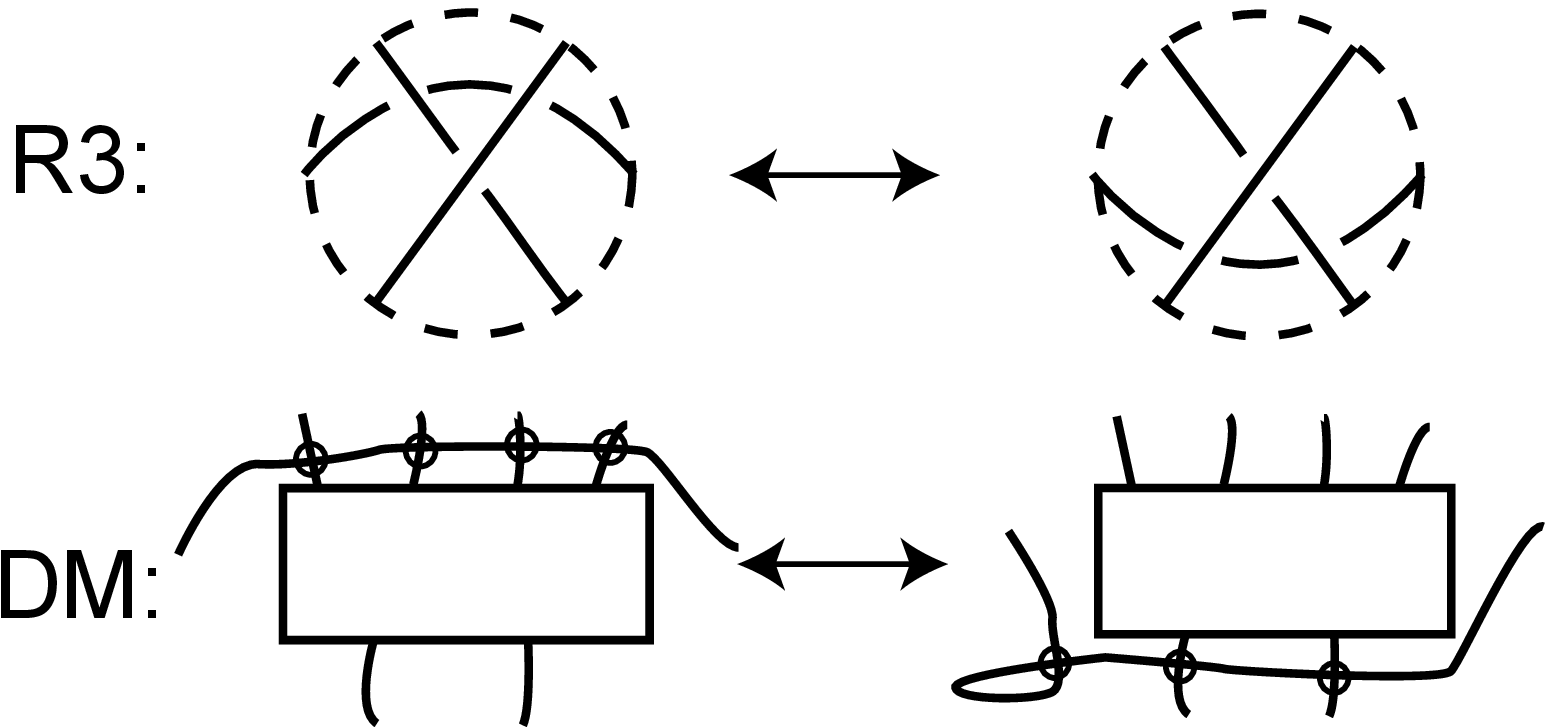}
\caption{Moves of virtual diagrams}\label{pic:virt_moves}
\end {figure}

A \emph{unicursal component} is a minimal set of diagram edges which is closed under passing from an edge to its opposite (at some end of the edge) edge. A \emph{ virtual knot} is an equivalence class of diagrams with one unicursal component.

On the other hand, virtual knots can be defined by means Gauss diagrams~\cite{GPV}. The \emph{Gauss diagram} $G=G(D)$ of a virtual knot diagram $D$ is a chord diagram whose chords correspond to the classical crossings of $D$, see Fig.~\ref{fig:virtual_gauss_diagram}. The chords carry an orientation (from over-crossing to under-crossing) and the sign of the crossings, see Fig.~\ref{fig:crossing_sign}.

\begin{figure}[h]
\centering\includegraphics[width=0.15\textwidth]{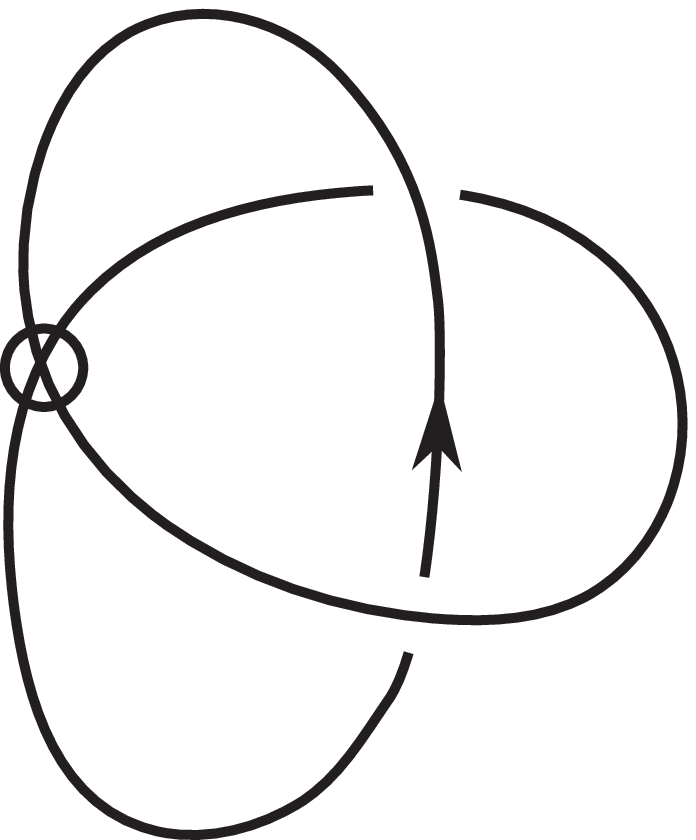}\qquad
\includegraphics[width=0.15\textwidth]{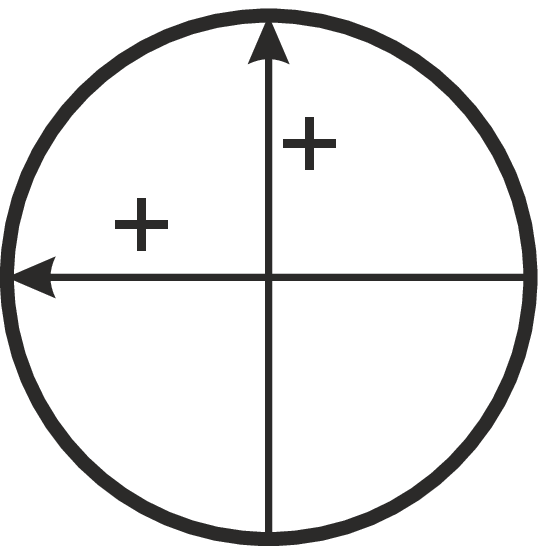}
\caption{The virtual trefoil and its Gauss diagram}\label{fig:virtual_gauss_diagram}
\end{figure}

\begin{figure}[h]
\centering\includegraphics[width=0.2\textwidth]{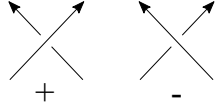}
\caption{The sign of a crossing}\label{fig:crossing_sign}
\end{figure}

Classical Reidemeister moves induce transformations of Gauss diagrams (see Fig.~\ref{fig:reidemeister_gauss}). Virtual knots are exactly the equivalence classes of Gauss diagrams modulo the induced Reidemeister moves.

\begin{figure}[h]
\centering\includegraphics[width=0.6\textwidth]{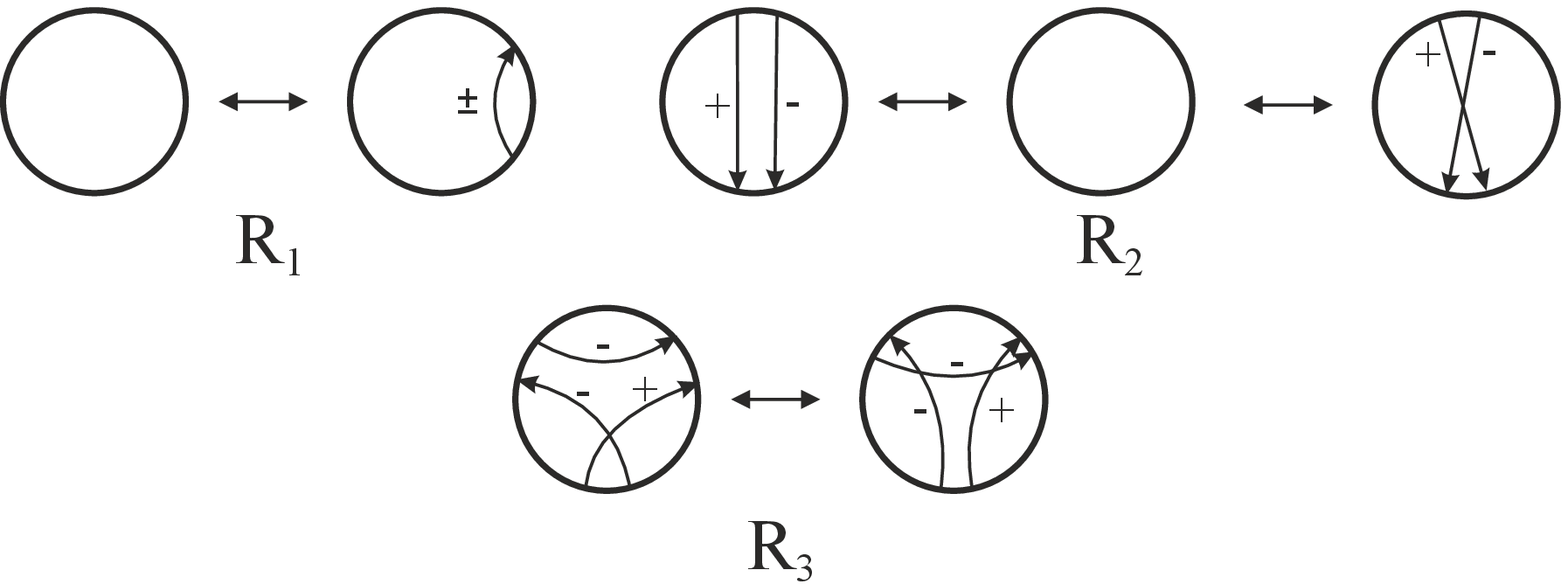}
\caption{Reidemeister moves on Gauss diagrams}\label{fig:reidemeister_gauss}
\end{figure}

Recall that an \emph{abstract knot diagram}~\cite{KK} can be constructed for a virtual knot diagram $D$ as follows. Consider a two-dimensional surface with boundary immersed in the neighbourhood of the diagram $D$ as shown in Fig.~\ref{fig:abstract_surface}. Then glue discs to the boundary of this surface. Denote the obtained closed surface by $AD(D)$.

\begin{figure}[h]
\centering\includegraphics[width=0.4\textwidth]{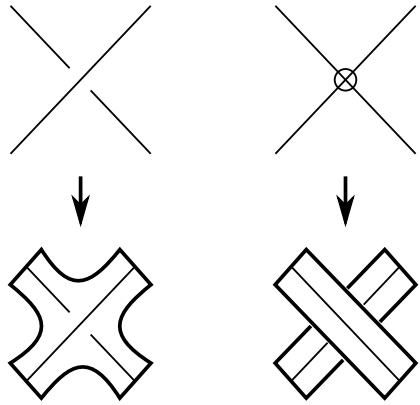}
\caption{Construction of the abstract knot diagram}\label{fig:abstract_surface}
\end{figure}

The virtual knot diagram $D$ lifts naturally to $AD(D)$ as its $1$-skeleton $\tilde D$. The graph $\tilde D$ can be considered as a knot diagram on the surface $AD(D)$ in with classical crossings. The pair $(AD(D),\tilde D)$ is the \emph{abstract knot diagram} of the virtual knot diagram $D$. We will call the graph $\tilde D$ the \emph{covering graph of the diagram $D$}.

The Gauss diagram $G(D)$ and the covering graph $\tilde D$ of a virtual knot diagram $D$ are connected by continuous maps (Fig.~\ref{fig:gauss_virtual_diagrams}). The projection $p_1\colon G(D)\to\tilde D$ contracts the chords of $G(D)$ to points, and the projection $p_2$ is an immersion of a neighbourhood of $\tilde D$ in the surface $AD(D)$ to the plane $\R^2$. The double points of the projection $p_2\colon\tilde D\to D$ corresponds to the virtual crossing of the diagram $D$.

\begin{figure}[h]
\centering\includegraphics[width=0.6\textwidth]{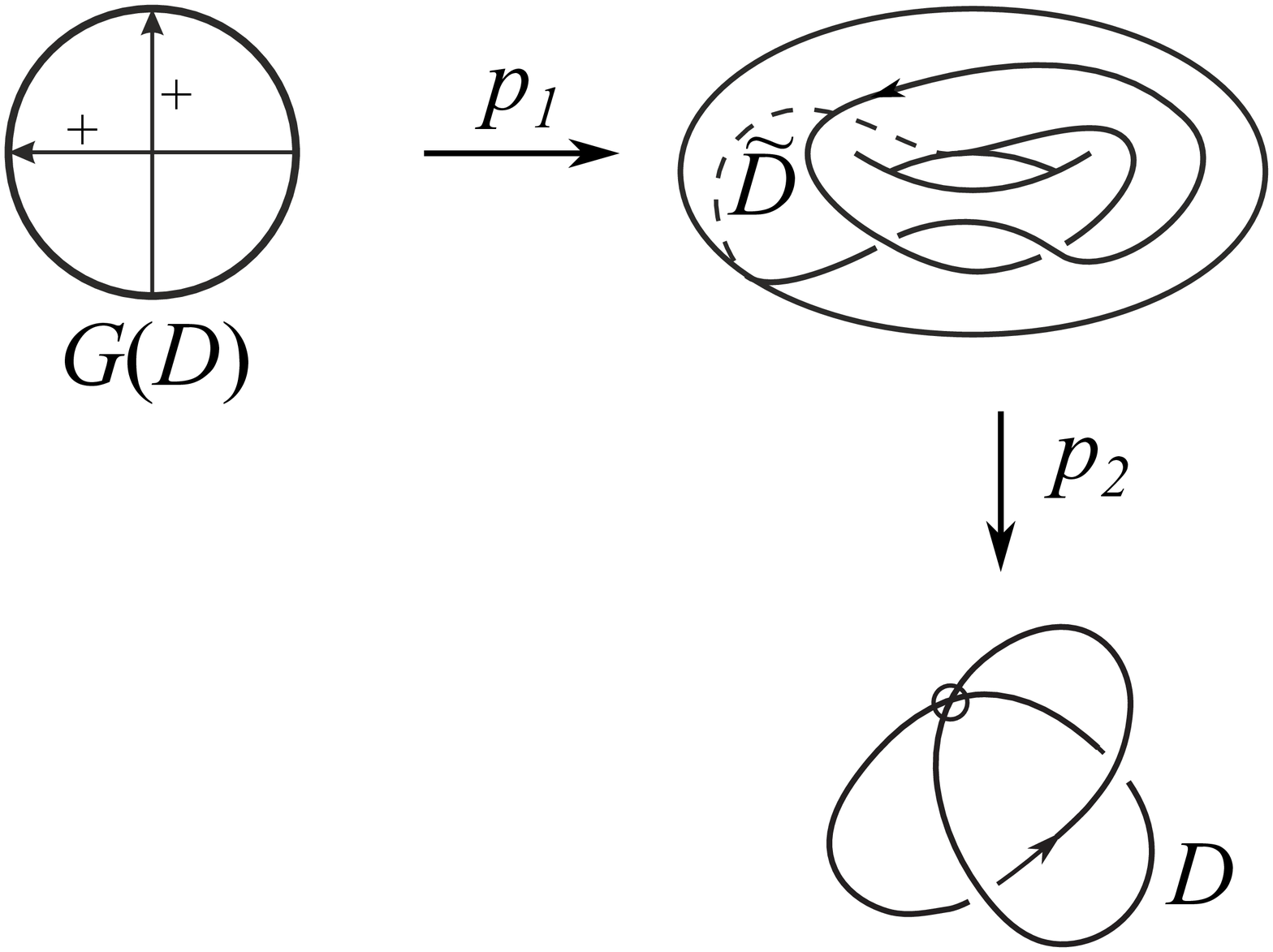}
\caption{A virtual diagram and its abstract knot diagram and Gauss diagram}\label{fig:gauss_virtual_diagrams}
\end{figure}

Since the surface $AD(D)$ is oriented, for any two cycles $\tilde C_1$, $\tilde C_2$ in the graph $\tilde D$ their intersection number $\tilde C_1\cdot\tilde C_2\in\Z$ is defined. Any two cycles $C_1,C_2$ in the diagram $D$ can be uniquely lifted to cycles $\tilde C_1,\tilde C_2$ in the covering graph $\tilde D$, so one can consider the intersection number $C_1\cdot C_2=\tilde C_1\cdot\tilde C_2$. The intersection number can be calculated in the diagram $D$ as the sum of signs of the intersection points of the cycles (Fig.~\ref{fig:intersection_sign}).

\begin{figure}[h]
\centering\includegraphics[width=0.4\textwidth]{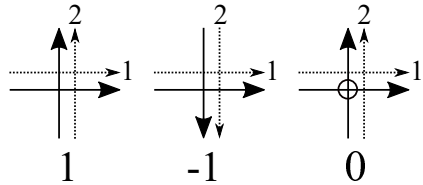}
\caption{Signs of intersection points of cycles in a virtual diagram}\label{fig:intersection_sign}
\end{figure}

\subsection{Parity on virtual knots}

Let $\mathcal K$ be an oriented virtual knot. Consider the set $\mathfrak K$ of the diagrams of the knot $K$. The set $\mathfrak K$ can be considered as the objects of a diagram category whose morphisms are compositions of isotopies and Reidemeister moves.

For any diagrams $D,D'\in\mathfrak K$ and a morphism $f\colon D\to D'$ between them, there is a correspondence $f_*\colon\V(D)\to\V(D')$ between the crossings of the diagrams. The map $f_*$ is a partial bijection between the crossing set because some crossing can disappear within the transformation $f$ and other crossing can appear by a first or second Reidemeister move.

Let $A$ be an abelian group. Let us recall the definition of oriented parity~\cite{N,N1}.

\begin{definition}\label{def:oriented_parity}
An \emph{oriented parity} $p$ is a family of maps $p_D\colon \V(D)\to G$ defined for each diagram $D$ of the knot $\mathcal K$, that possesses the following properties:
\begin{itemize}
\item[(P0)] for any morphism $f\colon D\to D'$ and any crossing $v\in\V(D)$ for which there exists the corresponding crossing $f_*(v)\in\V(D')$, one has $p_{D}(v)=p_{D'}(f_*(v))$;

\item[(P1)] if $f\colon D\to D'$ is a decreasing first Reidemeister move and $v\in\V(D)$ is the disappearing crossing then $p_D(v)=0$;

\item[(P2)] $p_D(v_1)+p_D(v_2)=0$ for any crossings $v_1,\,v_2$ to which a decreasing second Reidemeister move can be applied;

\item[(P3+)]
 if $f\colon D\to D'$ is a third Reidemeister move then
 \[
 \epsilon_\Delta(v_1)\cdot p_D(v_1)+\epsilon_\Delta(v_2)\cdot p_D(v_2)+\epsilon_\Delta(v_3)\cdot p_D(v_3)=0
 \]
  where $v_1,v_2,v_3$  are the crossings involved in the move and $\epsilon_\Delta(v_i)$ is the incidence index of the crossing $v_i$ to the disappearing triangle $\Delta$, see Fig.~\ref{pic:incidence_index}.

\begin {figure}[h]
\centering
\includegraphics[width=0.1\textwidth]{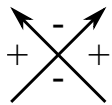}
\caption{Incidence indices}\label{pic:incidence_index}
\end {figure}
\end{itemize}
\end{definition}

\begin{remark}\label{rem:oriented_parity}
1. The properties (P1),(P2) and (P3+) can be unified in the following property: $\sum_i \epsilon_\Delta(v_i)\cdot p_D(v_i)=0$ if $f$ is a decreasing first or second Reidemeister move or a third Reidemeister move, $v_i$ are the crossings which take part in the move and
$\epsilon_\Delta(v_i)$ is the incidence index of the crossing to the disappearing region $\Delta$.

2. Oriented parities include parities with coefficients defined in~\cite{IMN} because those parities posses the property $2p(v)=0$ for any crossing $v$ so the incidence indices $\epsilon_\Delta(v)$ do not matter.
\end{remark}

Among examples of parities are the Gaussian parity, the link parity (with coefficients in $\Z_2$) of virtual knots, the index parity with coefficients in $\Z$ and the homological parity of knots in a given surface (with coefficients in the first homology group of the surface)~\cite{IMN1,N1}.

\begin{example}[Gaussian parity]
Let $D$ be a diagram of a knot $\mathcal K$ and $c\in\V(D)$ be a crossing of $D$. The crossing $c$ splits $D$ into two oriented halves (see Fig.~\ref{fig:knot_halves}). The Gaussian parity $gp_D(c)\in\Z_2$ of $c$ is the parity of the number of crossing points on any half.

\begin{figure}[h]
\centering\includegraphics[width=0.5\textwidth]{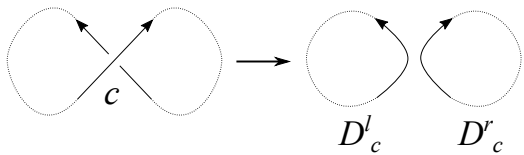}
\caption{The left and the right halves of the diagram}\label{fig:knot_halves}
\end{figure}

For example, a diagram of a virtual eight-knot in Fig.~\ref{fig:gaussian_parity_example} has two odd and one even crossings.

\begin{figure}[h]
\centering
\includegraphics[height=3cm]{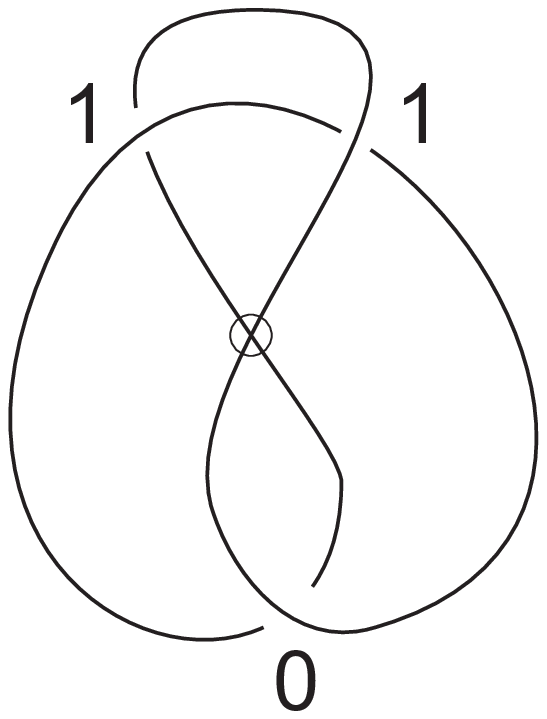}
\caption{Gaussian parity on a virtual knot diagram}\label{fig:gaussian_parity_example}
\end{figure}
\end{example}

\begin{example}[Link parity]\label{exa:link_parity}
Let $D$ be a diagram of some link $\mathcal L$ with two components. The crossings of the diagram $D$ can be divided into \emph{self-crossings} of a component of $\mathcal L$ and \emph{mixed crossings} where two different components intersect. One assigns the parity $0$ to the self-crossings and the parity $1$ to the mixed crossings.
\end{example}

\begin{example}[Index parity]
 Let $D$ be a diagram of an oriented virtual knot $\mathcal K$ and $v\in\V(D)$ be a crossing of $D$. Let us calculate the crossing points on the left half $D^l_v$ of the diagram at the crossing $v$ with the signs as in Fig.~\ref{fig:classical_crossing_sign}. The sum is the \emph{index parity} $ip_D(v)$ of the crossing $v$. Note that $ip_D(v)$ coincides with the intersection index of~\cite{Henrich} and differs by the sign of the crossing $v$ from the index $W_K(v)$ from~\cite{K2}: $ip_D(v)=sgn(v)\cdot W_K(v)$ (see also~\cite{Cheng}).

  \begin{figure}[h]
\centering\includegraphics[width=0.35\textwidth]{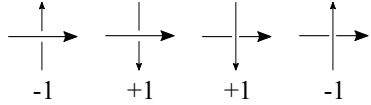}
\caption{Intersection number of a classical crossing}\label{fig:classical_crossing_sign}
\end{figure}

The Gaussian parity is the $\Z_2$ reduction of the index parity.

Since the index parity does not use under-overcrossing structure, this parity is defined for flat knots.
\end{example}

\begin{example}[Homological parity]
Let $\mathcal K$ be a knot in a fixed oriented surface $S$ and $D\subset S$ be its diagram. Any crossing $v\in\V(D)$ splits the diagram into left and right halves. This halves are cycles in the surface. The \emph{homological parity} of the crossing $v$ is the element $p^h_D(v)=[D^l_v]\in H_1(S,\Z)/[\mathcal K]$ where $[\mathcal K]$ is the homology class of the knot $\mathcal K$.
\end{example}

Let us remind basic properties of oriented parity~\cite{N1}.

 \begin{proposition}\label{prop:bigon_trigon_parity_property}
 Let $p$ be an oriented parity on diagrams of a knot $\mathcal K$ and $D$ be a diagram of $\mathcal K$.

 1. Assume that crossings $v_1,v_2\in\V(D)$ form a bigon (Fig.~\ref{pic:bigon_trigon} left). Then $p_D(v_1)+p_D(v_2)=0$;

 2. Assume that crossings $v_1,v_2,v_3\in\V(D)$ form a triangle $\Delta$ (Fig.~\ref{pic:bigon_trigon} right). Then $\epsilon_\Delta(v_1)\cdot p_D(v_1)+\epsilon_\Delta(v_2)\cdot p_D(v_2)+\epsilon_\Delta(v_3)\cdot p_D(v_3)=0$.

\begin {figure}[h]
\centering
\raisebox{-0.5\height}{\includegraphics[width=0.2\textwidth]{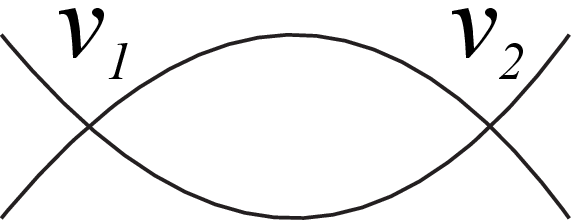}}\qquad
\raisebox{-0.5\height}{\includegraphics[width=0.15\textwidth]{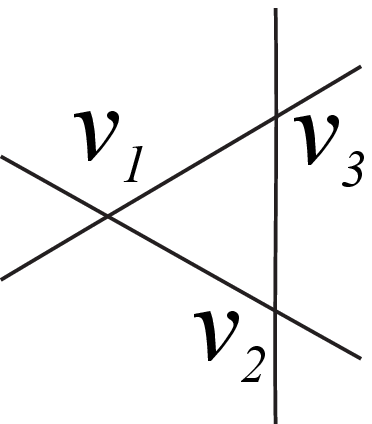}}
\caption{A bigon and a triangle. The under-overcrossing structure does not matter.}\label{pic:bigon_trigon}
\end {figure}
 \end{proposition}

\begin{proof}
1. We need only to prove the statement for alternating bigons. Apply a first Reidemeister move $f$ as shown in Fig.~\ref{fig:alt_bigon}.  By the properties  (P1) and (P3+) we have $$\epsilon(v_1)p_{D'}(v_1)+\epsilon(v_2)p_{D'}(v_2)+\epsilon(w)p_{D'}=0$$
 and $p_{D'}(w)=0$. Hence, $\epsilon(v_1)p_{D'}(v_1)+\epsilon(v_2)p_{D'}(v_2)=0$. Since $\epsilon(v_1)=\epsilon(v_2)$, we have $p_{D'}(v_1)+p_{D'}(v_2)=0$ and $p_{D}(v_1)+p_{D}(v_2)=0$.

\begin {figure}[h]
\centering
\includegraphics[width=0.4\textwidth]{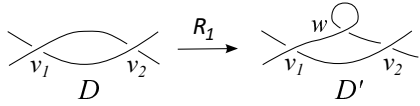}
\caption{Proof for an alternating bigon}\label{fig:alt_bigon}
\end {figure}

2. Let $v_1,v_2,v_3$ be the crossings of a triangle in the diagram $D$. If a third Reidemeister move can be applied to the triangle, the equality follows from the property (P3+).

\begin {figure}[h]
\centering
\includegraphics[width=0.5\textwidth]{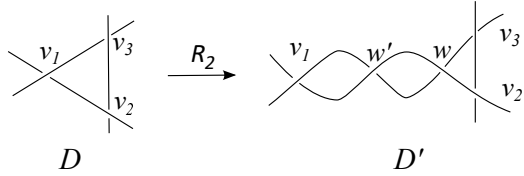}
\caption{Proof for an alternating triangle}\label{fig:alt_trigon}
\end {figure}

If $v_1,v_2,v_3$ form an alternating triangle, apply a second Reidemeister move $f\colon D\to D'$ as shown in Fig.~\ref{fig:alt_trigon}. Then $\epsilon(v_1)=\epsilon(w)$. By the property (P2) and the previous statement, $p_{D'}(v_1)+p_{D'}(w')=p_{D'}(w)+p_{D'}(w')=1$, hence $p_{D'}(v_1)=p_{D'}(w)$. By the property (P3+) the equality, we have
\[
\epsilon(v_2)p_{D'}(v_2)+\epsilon(v_3)p_{D'}(v_3)+\epsilon(w)p_{D'}(w)=0.
\]
Thus, $\epsilon(v_1)p_{D'}(v_1)+\epsilon(v_2)p_{D'}(v_2)+\epsilon(v_3)p_{D'}(v_3)=0$ and
\[
\epsilon(v_1)p_{D}(v_1)+\epsilon(v_2)p_{D}(v_2)+\epsilon(v_3)p_{D}(v_3)=0.
\]
\end{proof}


\section{Potentials and parity cycle}\label{sect:potentials}

Let $p$ be a parity with coefficients in an abelian group $A$ on diagrams of an oriented virtual knot $\mathcal K$. Let $\mathfrak K$ be the category of diagrams of the knot $\mathcal K$.

Let $D$ be a virtual diagram of $\mathcal K$. An \emph{arc} of the diagram $D$ is a part of $D$ which starts end ends at a classical crossing and contains no classical crossings inside. Denote the set of arcs of the diagram $D$ by $\mathcal A(D)$. The set $\mathcal A(D)$ can be identified with the set of edges of $1$-skeleton $\tilde D$ on the abstract knot diagram $AD(D)$ which corresponds to $D$, or with the set of edges of the core circle of the Gauss diagram $G(D)$.

An isotopy or a detour move $f\colon D\to D'$ induces a bijection $f_*\colon\mathcal A(D)\to\mathcal A(D')$. A Reidemeister move $f\colon D\to D'$ induces a correspondence between the arcs: an arc $a'\in\mathcal A(D')$ corresponds to an arc $a\in\mathcal A(D)$ if $a$ and $a'$ have a common point outside the area where the move occurs. For example, for an increasing first Reidemeister move, some arc $a\in\mathcal A(D)$ corresponds to two arcs $a',a''\in\mathcal A(D')$ adjacent to the loop (Fig.~\ref{fig:correspondent_arcs}). An increasing second Reidemeister move splits some two arcs $a$ and $b$ into arcs $a'$ and $a''$ correspondent to $a$ and arcs $b'$ and $b''$ correspondent to $b$. The small arcs $c,d$ participating in Reidemeister moves have no correspondent arcs.

\begin{figure}[h]
\centering
\includegraphics[width=0.4\textwidth]{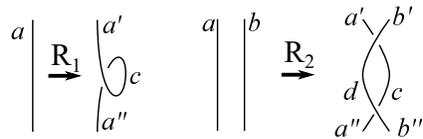}
\caption{Correspondent arcs}\label{fig:correspondent_arcs}
\end{figure}

\subsection{Potentials}

\begin{definition}
Let $a,b\in\mathcal A(D)$ be two arcs of the diagram $D$. The \emph{potential $\delta_{a,b}$ between the arcs $a$ and $b$} is defined as the parity $p_{D'}(v)\in A$ of the crossing $v$ produced by the second Reidemeister move $D\to D'$ applied to $a$ and $b$ (Fig.~\ref{fig:potentials}). By Proposition~\ref{prop:bigon_trigon_parity_property}, the value $\delta_{a,b}$ does not depend on which arcs goes over in the Reidemeister move. Analogously, we define the potentials $\delta_{b,a},\delta_{a,\bar b},\delta_{\bar b,a},\delta_{\bar a,\bar b},\delta_{\bar b,\bar a}$ (Fig.~\ref{fig:potentials}).

\begin{figure}[h]
\centering
\includegraphics[width=0.8\textwidth]{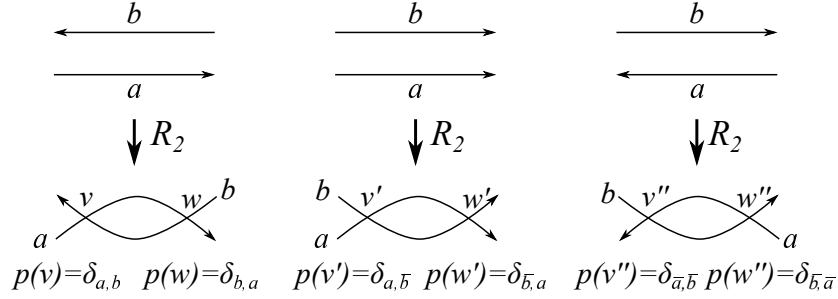}
\caption{Potentials between arcs}\label{fig:potentials}
\end{figure}

For any arc $a\in\mathcal A(D)$ we denote $\delta_a=\delta_{a,\bar a}$ (Fig.~\ref{fig:potential_delta}).

\begin{figure}[h]
\centering
\includegraphics[width=0.5\textwidth]{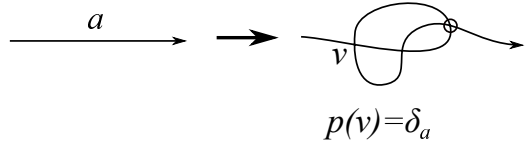}
\caption{The potential $\delta_a$}\label{fig:potential_delta}
\end{figure}
\end{definition}

\begin{proposition}\label{prop:potential_properties}
\begin{enumerate}
\item For any arcs $a,b\in\mathcal A(D)$ we have
\[
\delta_{b,a}=-\delta_{a,b},\qquad \delta_{\bar b,a}=-\delta_{a,\bar b}, \qquad
\delta_{\bar b,\bar a}=-\delta_{\bar a,\bar b}
\]

\item For any arcs $a,b,c\in\mathcal A(D)$ we have
\begin{gather*}
\delta_{a,b}+\delta_{b,c}+\delta_{c,a}=0,\qquad
\delta_{\bar a,b}-\delta_{b,c}+\delta_{c,\bar a}=0,\\
\delta_{a,\bar b}-\delta_{\bar b,\bar c}+\delta_{\bar c,a}=0,\qquad
\delta_{\bar a,\bar b}+\delta_{\bar b,\bar c}+\delta_{\bar c,\bar a}=0.
\end{gather*}
\end{enumerate}
\end{proposition}

\begin{proof}
1. The first statement follows from the definition of potentials and property (P2).

2. Let us prove the identity $\delta_{a,b}+\delta_{b,c}+\delta_{c,a}=0$. Apply second Reidemeister moves to the arcs $a,b,c$ as shown in Fig.~\ref{fig:poten_add}. Denote the obtained diagram by $D'$. Then $p_{D'}(u)=\delta_{a,b}, p_{D'}(v)=\delta_{b,c}, p_{D'}(w)=\delta_{c,a}$. By property (P3+) we have $\epsilon(u)p_{D'}(u)+\epsilon(v)p_{D'}(v)+\epsilon(w)p_{D'}(w)=0$. Since $\epsilon(u)=\epsilon(v)=\epsilon(w)=+1$, we get the required equality.

\begin{figure}[h]
\centering\includegraphics[width=0.5\textwidth]{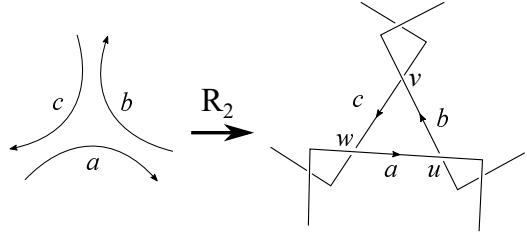}
\caption{Proof of $\delta_{a,b}+\delta_{b,c}+\delta_{c,a}=0$}\label{fig:poten_add}
\end{figure}
The other equalities can be proved analogously.
\end{proof}

\begin{corollary}\label{cor:potential_property}
For any arcs $a,b\in\mathcal A(D)$ we have $\delta_{a,\bar b}=\delta_{b,a}+\delta_b$, $\delta_{\bar a,b}=\delta_{b,a}-\delta_a$, $\delta_{\bar a,\bar b}=\delta_{a,b}+\delta_a-\delta_b$.
\end{corollary}

\begin{proposition}\label{prop:parity_potential}
Let $a,b,c,d$ be the arcs incident to a crossing $v$ of the diagram $D$ as shown in Fig.~\ref{fig:potential_parity}. Then $p_D(v)=\delta_{a,d}=\delta_{b,\bar a}=\delta_{\bar c,\bar b}=\delta_{\bar d, c}$.

\begin{figure}[h]
\centering\includegraphics[width=0.12\textwidth]{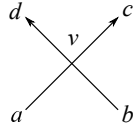}
\caption{A crossing}\label{fig:potential_parity}
\end{figure}
\end{proposition}

\begin{proof}
Apply a second Reidemeister move as shown in Fig.~\ref{fig:potential_parity_proof}. Let $D'$ be the obtained diagram. By property (P2) $p_{D'}(v)=-p_{D'}(w')=p_{D'}(w)$, and by definition $p_{D'}(w)=\delta_{a,d}$. Then $p_D(v)=p_{D'}(v)=\delta_{a,d}$. The other equalities are proved analogously.

\begin{figure}[h]
\centering\includegraphics[width=0.5\textwidth]{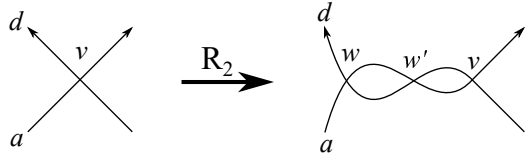}
\caption{Proof of the equality $p(v)=\delta_{a,d}$}\label{fig:potential_parity_proof}
\end{figure}
\end{proof}

\begin{corollary}\label{cor:parity_delta_potential}
Let $a,b,c,d$ be the arcs incident to a crossing $v$ of the diagram $D$ as shown in Fig.~\ref{fig:potential_parity}. Then $\delta_a+\delta_b=\delta_c+\delta_d$.
\end{corollary}
\begin{proof}
By  Corollary~\ref{cor:potential_property} and Propositions~\ref{prop:parity_potential}
\[
p_D(v)=\delta_{a,d}=\delta_a-\delta_{b,a}=\delta_{c,b}+\delta_b-\delta_c=-\delta_d-\delta_{d, c}.
\]
Then
\begin{multline*}
0=\delta_{a,d}-(\delta_a-\delta_{b,a})+(\delta_{c,b}+\delta_c-\delta_b)-(-\delta_d-\delta_{d, c})=\\
(\delta_{a,d}+\delta_{d, c}+\delta_{c,b}+\delta_{b,a})+(\delta_c+\delta_d-\delta_a-\delta_b)=\delta_c+\delta_d-\delta_a-\delta_b
\end{multline*}
since $\delta_{a,d}+\delta_{d, c}+\delta_{c,b}+\delta_{b,a}=\delta_{a,a}=0$.
\end{proof}

\subsection{Parity cycle}\label{subsect:parity_cycle}

Consider the cell complex $C_*(\tilde D,A)$ of the covering graph $\tilde D$ of the diagram $D$. The 0-chain space $C_0(\tilde D,A)$  can be identified with the group $A\V(D)$ which consists of linear combinations of crossings of the diagram $D$ with coefficients in the group $A$. The group of 1-chains $C_1(\tilde D,A)$ is identified with the group $A\mathcal A(D)$ of linear combinations of arcs of $D$ with coefficients in $A$. Having in mind these identifications, we shall write below $C_i(D,A)$ instead of $C_i(\tilde D,A)$, $i=0,1$.

Note that $H_1(D,A)=H_1(\tilde D,A)$ is the set of 1-cycles in the diagram $D$ with values in the group $A$.

\begin{definition}\label{def:invariant_normalized_cycles}
Let $\gamma_D\in C_1(D,A)$, $D\in\mathfrak K$, be a family of 1-chains assigned to all diagrams $D$ of the virtual knot $\mathcal K$. The family of chains $\gamma_D$ is \emph{invariant} if for any morphism (an isotopy or a Reidemeister move) $f\colon D\to D'$ and any correspondent arcs $a\in\mathcal A(D)$ and $a'\in\mathcal A(D')$ their coefficients coincide: $\gamma_D(a)=\gamma_{D'}(a')$.

A cycle $\delta_D\in H_1(D,A)$ is \emph{normalized} if the intersection number $D\cdot\delta_D=0$.
\end{definition}

The next theorem shows that any parity determines an invariant normalized cycle.

\begin{definition}\label{def:parity_cycle}
Let $D$ be a diagram of the oriented virtual knot $\mathcal K$, and $\tilde D$ be the covering graph of the diagram $D$. For an oriented parity $p$ with coefficients in an abelian group $A$ we define a 1-chain $\delta^p_D$ as follows
\begin{equation}\label{eq:parity_cycle}
\delta_D^p=\sum_{a\in\mathcal A(D)} \delta_a\cdot a \in C_1(\tilde D,A).
\end{equation}

The 1-chain $\delta_D^p$ is called the \emph{parity cycle} of the oriented parity $p$ in the diagram $D$.
\end{definition}

\begin{theorem}\label{thm:parity_cycle}
Let $p$ be an oriented parity with coefficients in an abelian group $A$ on the diagrams of an oriented virtual knot $\mathcal K$, $D$ be a diagram of $\mathcal K$, and $\delta_D^p$ be the parity cycle. Then
\begin{enumerate}
\item $\delta_D^p$ is a genuine 1-cycle in the cell complex $C_*(\tilde D,A)$, i.e. $d\delta_D^p=0$;
\item $\delta^p_D$ is an invariant cycle;
\item $\delta^p_D$ is normalized;
\item (the intersection formula) for any crossing $v\in\V(D)$ we have
\begin{equation}\label{eq:intersection_formula}
p_D(v)=D^l_v\cdot\delta_D^p.
\end{equation}
\end{enumerate}
\end{theorem}


\begin{proof}

1) The first statement of the theorem is a reformulation of Corollary~\ref{cor:parity_delta_potential}.

2) The invariance of $\delta_D^p$ is a consequence of the property (P0) of the parity $p$. Indeed, let $f\colon D\to D'$ be a morphism of diagrams and $a\in \mathcal A$ and $a'\in\mathcal A(D')$ be correspondent arcs . Apply second Reidemeister moves as shown in Fig.~\ref{fig:potential_delta} to the common part of the arcs. Then we get a new diagram $D_1$ from $D$ and a diagram $D'_1$ from $D'$, and crossings $v\in\V(D_1)$ and $v'\in\V(D_1')$ such that $\delta_a=p_{D_1}(v)$ and $\delta_{a'}=p_{D'_1}(v')$. The morphism $f$ induces a diagram morphism $f_1\colon D_1\to D'_1$ which identifies the crossings $v$ and $v'$. Hence, by the property (P0) we have $p_{D_1}(v)=p_{D'_1}(v')$ and $\delta_a=\delta_{a'}$.

3) Let us prove first the intersection formula. Let $O\in D\setminus\V(D)$ be a point of the diagram which is not a crossing, and $o\in\mathcal A(D), O\in o$, be the corresponding arc of the diagram. For any arc $a\in\mathcal A(D)$ denote $\phi_a=\delta_{a,o}$. By Proposition~\ref{prop:potential_properties}, for any $a,b\in\mathcal A(D)$ the equalities hold
\[
\delta_{a,b}=\phi_a-\phi_b,\quad \delta_{\bar a,b}=\phi_b-\phi_a-\delta_a,\quad \delta_{a,\bar b}=\phi_b-\phi_a+\delta_b,\quad \delta_{\bar a,\bar b}=\phi_a-\phi_b+\delta_a-\delta_b.
\]

Let $v$ be a crossing of the diagram $D$ (see Fig.~\ref{fig:potential_parity}). Then by Proposition~\ref{prop:parity_potential} we have
\[
p_D(v)=\delta_{a,d}=\delta_a-\delta_{b,a}=\delta_{c,b}+\delta_b-\delta_c=-\delta_d-\delta_{d, c}.
\]
Then
\[
\phi_c-\phi_a=\delta_{c,a}=\delta_{c,d}+\delta_{d,a}=-\delta_{d,c}-\delta_{a,d}=\delta_d
\]
and
\[
\phi_d-\phi_b=\delta_{d,b}=\delta_{d,a}+\delta_{a,b}=-\delta_{a,d}-\delta_{b,a}=-\delta_a.
\]

Thus, the potential $\phi$ changes by $\pm\delta$ when passing a crossing (Fig.~\ref{fig:change_potential}).

\begin{figure}[h]
\centering\includegraphics[width=0.5\textwidth]{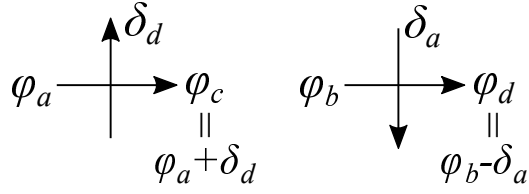}
\caption{Change of the potential $\phi$ by a crossing}\label{fig:change_potential}
\end{figure}

Now, fix a crossing $v$ (Fig.~\ref{fig:potential_parity}). The parity $p_D(v)$ is equal to $\phi_a-\phi_d$ by Proposition~\ref{prop:parity_potential}. The arcs $a$ and $d$ are the end and the start of the left half $D_v$. Hence, the difference between the potentials $\phi_a$ and $\phi_d$ is the sum
\[
\phi_a-\phi_d=\sum_{v'\in D^l_v}\Delta\phi(v')
\]
of the increments $\Delta\phi(v')$ of the potential at the crossing in the left half of the diagram.

On the other hand, the intersection number $D^l_v\cdot\delta_D^p$ is the sum of signs of the intersection points of the cycles multiplied by the values of the cycles. Shift the half $D^l_v$ to the left of $D$. Then the intersection points correspond to the crossings in the half $D^l_v$ (Fig.~\ref{fig:intersection_formula_proof}) and the contribution of the point corresponding to a crossing $v'\in D^l_v$ is equal to $\Delta\phi(v')$. Thus, $$p_D(v)=\phi_a-\phi_d=D^l_v\cdot\delta_D^p.$$

\begin{figure}[h]
\centering\includegraphics[width=0.6\textwidth]{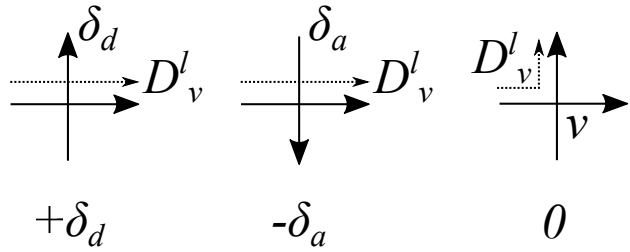}
\caption{Intersection points of $D^l_v$ and $D$}\label{fig:intersection_formula_proof}
\end{figure}

4) Consider $D$ as a route along the diagram which starts and ends at the point $O$. The reasonings above show that $D\cdot\delta_D^p=\phi_O-\phi_O=0$.
\end{proof}

Let us prove the inverse statement.

\begin{theorem}\label{thm:parity_cycle_inverse}
  Let $\delta_D\in H_1(D,A)$, $D\in\mathfrak K$, be an invariant normalized cycle with values in an abelian group $A$ on the diagrams of a virtual knot $\mathcal K$. Then the intersection formula
\[
p^\delta_D(v)=D^l_v\cdot\delta_D
\]
defines an oriented parity $p^\delta$ on diagrams of the knot $\mathcal K$.
\end{theorem}

\begin{proof}
  Let us check the properties of oriented parity.

(P0) Let $f\colon D\to D'$ be a Reidemeister move, and $v\in\V(D)$ and $v'\in\V(D')$ be correspondent crossings. The coverings graph $\tilde D$ and $\tilde D'$ can be considered as embedded graphs in the surface $AD(D)$ (or $AD(D')$ when $f$ is an increasing second Reidemeister move). The cycles $D^l_v,D'^l_{v'},\delta_D$ and $\delta_{D'}$ define elements $[D^l_v]=[\tilde D^l_v]$ and $[D'^l_v]=[\tilde D'^l_{v'}]$ in $H_1(AD(D),\Z)$ and $[\delta_D],[\delta_{D'}]\in H_1(AD(D),A)$. The intersection values are the results of the intersection map
\[
H_1(AD(D),\Z)\times H_1(AD(D),A)\to H_0(AD(D),A)=A.
\]
Consider the continuous map $h\colon AD(D)\to AD(D)$ which contacts the region where the Reidemeister move occurs to a point (Fig.~\ref{fig:reidemeister_intermediate}). Then $h(\tilde D^l_v)=h(\tilde D'^l_{v'})$. Since $h$ is homotopic to the identity map, then $[D^l_v]=[h(\tilde D^l_v)]=[\tilde D'^l_v]$, and $[\delta_D]=h_*[\delta_D]$ and $[\delta_{D'}]=h_*[\delta_{D'}]$. On the other hand, $h_*[\delta_D]=h_*[\delta_{D'}]$ because the chain $\delta_D$ is invariant. Thus,
\[
p^\delta_D(v)=D^l_v\cdot\delta_D=[D^l_v]\cdot[\delta_D]=[D'^l_{v'}]\cdot[\delta_{D'}]=D'^l_{v'}\cdot\delta_{D'}=p^\delta_{D'}(v').
\]

\begin{figure}[h]
\centering\includegraphics[width=0.5\textwidth]{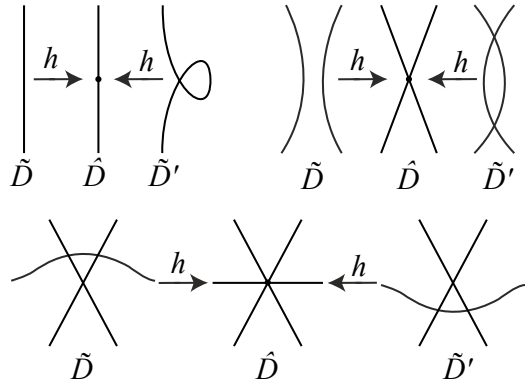}
\caption{The contraction map $h$}\label{fig:reidemeister_intermediate}
\end{figure}

(P1) Let $v\in\V(D)$ be a crossing participating in a first Reidemeister move. Then the diagram half $D^l_v$ at the crossing $v$ is either homotopic to $D$ or contractible in $AD(D)$ (Fig.~\ref{fig:half_r1}). In both cases, $D^l_v\cdot\delta_D=0$ because $\delta_D$ is normalized.

\begin{figure}[h]
\centering\includegraphics[width=0.6\textwidth]{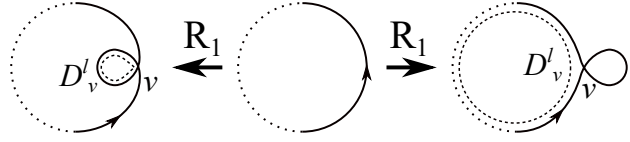}
\caption{Left half at a crossing in a first Reidemeister move}\label{fig:half_r1}
\end{figure}

(P2) Let $u,v\in\V(D)$ be crossings participating in a second Reidemeister move. Then the diagram $D$ is homologous to the sum of the halves $D^l_u$ and $D^l_v$ in $AD(D)$ (Fig.~\ref{fig:half_r2}).

\begin{figure}[h]
\centering\includegraphics[width=0.18\textwidth]{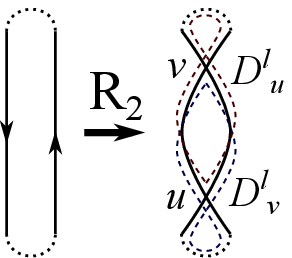}\qquad
\includegraphics[width=0.18\textwidth]{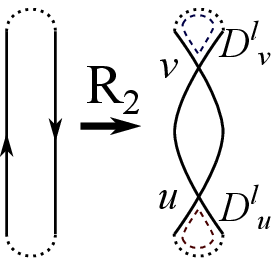}\qquad\includegraphics[width=0.27\textwidth]{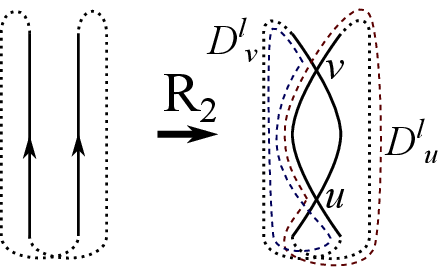}
\caption{Left halves at crossings in a second Reidemeister move}\label{fig:half_r2}
\end{figure}

Then $p^\delta_D(u)+p^\delta_D(v)=D^l_u\cdot\delta_D+D^l_v\cdot\delta_D=D\cdot\delta_D=0$.

(P3) Let $u,v,w\in\V(D)$ be crossings participating in a third Reidemeister move. If the incidence indices of the crossings coincide then the sum $[D^l_u]+[D^l_v]+[D^l_w]$ is equal either to $[D]$ or to $2[D]$ in $H_1(AD(D),\Z)$ (see Fig.~\ref{fig:half_r3} top). If $\epsilon(u)=1$ and $\epsilon(v)=\epsilon(w)=-1$ then $[D^l_u]-[D^l_v]-[D^l_w]$ is equal either to $-[D]$ or  $0$ (see Fig.~\ref{fig:half_r3} bottom).

\begin{figure}[h]
\centering\includegraphics[width=0.5\textwidth]{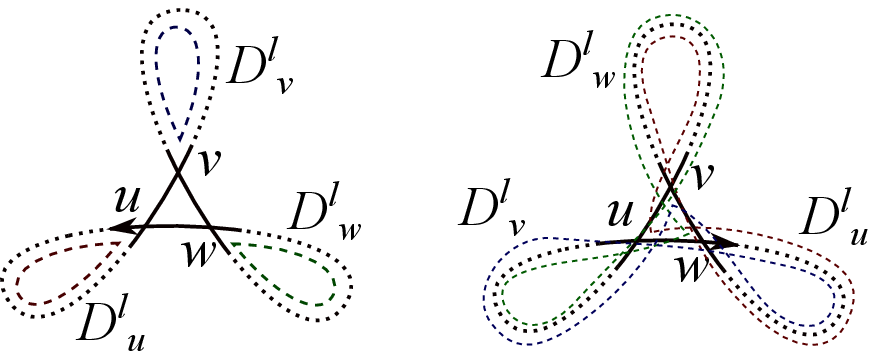}\\ \includegraphics[width=0.5\textwidth]{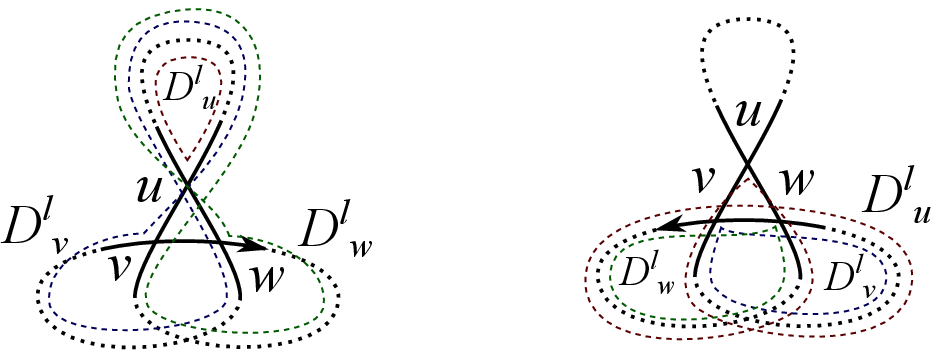}
\caption{Left halves at crossings in a third Reidemeister move}\label{fig:half_r3}
\end{figure}

In all cases \begin{multline*}
\epsilon(u)p^\delta_D(u)+\epsilon(v)p^\delta_D(v)+\epsilon(w)p^\delta_D(w)=\\
(\epsilon(u)D^l_u+\epsilon(v)D^l_v+\epsilon(w)D^l_w)\cdot\delta_D=kD\cdot\delta_D=0,
\end{multline*}
$k=-1,0,1,2$.

\end{proof}

Thus, an oriented parity defines an invariant normalized cycle, and an invariant normalized cycle generates a parity. In fact, these correspondences are inverse to each other.

\begin{proposition}\label{prop:parity_cycle_bijection}
1. For any oriented parity $p$ the parity $p^{\delta_p}$ coincides with $p$.

2. For any invariant normalized 1-cycle $\delta$ the parity cycle $\delta_{p^\delta}$ coincides with $\delta$.
\end{proposition}

\begin{proof}
The first statement follows from the intersection formula~\eqref{eq:intersection_formula} and the definition of the parity $p^{\delta_p}$.

For the second statement, take an arbitrary arc $a\in\mathcal A(D)$. Then
\[
(\delta_{p^\delta})_a=p^\delta(v)=D^l_v\cdot\delta=\delta_a
\]
(see Fig.~\ref{fig:delta_marked}). Thus cycles $\delta$ and $\delta_{p^\delta}$ coincide.

\begin{figure}[h]
\centering\includegraphics[width=0.6\textwidth]{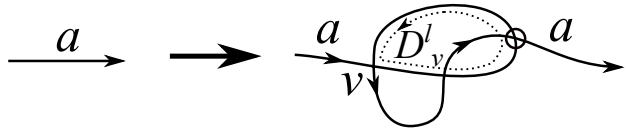}
\caption{Proof of the equality $(\delta_{p^\delta})_a=\delta_a$}\label{fig:delta_marked}
\end{figure}
\end{proof}

\begin{remark}\label{rem:parity_cycle_isomorphism}
We can reformulate Proposition~\ref{prop:parity_cycle_bijection} as follows. Given a virtual knot $\mathcal K$ and an abelian group $A$, let $\mathcal{P}(\K,A)$ be the set of oriented parities with coefficients in $A$ on the diagrams of $\mathcal K$ and $\mathcal{NIC}(\K,A)$ be the set of normalized invariant cycles with coefficients in $A$ on the diagrams of $\mathcal K$. The sets $\mathcal{P}(\K,A)$ and $\mathcal{NIC}(\K,A)$ have natural structure of abelian groups. Then the intersection formula defines an isomorphism
\[
\mathcal{P}(\K,A)\simeq\mathcal{NIC}(\K,A).
\]
\end{remark}

\begin{example}[constant parity cycle]
Let $c\in A$ be an arbitrary element. Define a constant 1-chain $\delta$: $\delta_a\equiv c$ for all arcs $a\in\mathcal A(D)$. Then $\delta$ is an invariant cycle. The cycle $\delta$ is normalized because $D\cdot\delta=cD\cdot D=0$. The parity determined by the cycle $\delta$ is proportional to the index parity:
\[
p^\delta_D(v)=c D^l_v\cdot D= -c\cdot ip_D(v),\ v\in\V(D).
\]
We have the minus in the formula because the signs in the definition of $ip_D$ are opposite to the signs of intersection points (cf. Fig.~\ref{fig:intersection_sign} and~\ref{fig:classical_crossing_sign}).
\end{example}

\begin{definition}\label{def:signature_invariant_cycle}
Let $\delta_D\in H_1(D,A)$, $D\in\mathfrak K$, be an invariant cycle with values in the group $A$. The \emph{signature} $\sigma(\delta_D)$ of the cycle $\delta_D$ is defined by the formula
\begin{equation}\label{eq:signature_invariant_cycle}
  \sigma(\delta_D)=D\cdot\delta_D\in A.
\end{equation}
\end{definition}

\begin{proposition}\label{prop:signature_invariant_cycle}
The signature $\sigma(\delta_D)$ of an invariant cycle $\delta_D$ is a knot invariant.
\end{proposition}

\begin{proof}
Let $f\colon D\to D'$ be a Reidemeister move. By reasonings of Theorem~\ref{thm:parity_cycle_inverse}, $[D']=[D]$ and $[\delta_D]=[\delta_{D'}]$ in $H_1(AD(D))$ (or in $H_1(AD(D'))$ if $f$ is an increasing second Reidemeister move). Hence,
$$D\cdot\delta_D =[D]\cdot[\delta_D]=[D']\cdot[\delta_{D'}]=D'\cdot\delta_{D'},$$
and the signature $\sigma(\delta_D)$ is invariant under Reidemeister moves.
\end{proof}

\begin{corollary}
  Let $\delta_D$ be an invariant cycle. If $D\cdot\delta_D=0$ for some diagram $D$ of the knot $\mathcal K$ then $D'\cdot\delta_{D'}=0$ for any other diagram $D'\in\mathfrak K$, i.e. the cycle is normalized.
\end{corollary}

The following statement shows that any invariant cycle defines a normalized invariant cycle.

\begin{corollary}\label{cor:cycle_normalization}
  Let $\delta$ be an invariant cycle with coefficients in an abelian group $A$. Denote $\bar A=A/\langle\sigma(\delta)\rangle$ be the factor-group by the cyclic subgroup generated by the signature, and $p\colon A\to\bar A$ be the canonical projection. Then $\bar\delta=p\circ\delta$ is an invariant normalized cycle with coefficients in $\bar A$.
\end{corollary}

\subsection{Parity cycle induced by a biquandle 1-cocycle}\label{subsect:cycle_biquandle}
Recall that a \emph{biquandle}~\cite{EN} is a set $B$ with two operations $\ast,\circ\colon B\times B\to B$ which satisfy the following conditions:
\begin{enumerate}
\item $x\circ x=x\ast x$ for all $x\in B$;
\item the maps $(x,y)\mapsto (y,x\circ y)$, $(x,y)\mapsto (x,y\ast x)$ and $(x,y)\mapsto (x\circ y, y\ast x)$ are bijections of $B\times B$
\item for any $x,y,z\in B$;
\begin{gather*}
(x\circ y)\circ (z\circ y)=(x\circ z)\circ (y\ast z),\\
(x\circ y)\ast (z\circ y)=(x\ast z)\circ (y\ast z),\\
(x\ast y)\ast (z\ast y)=(x\ast z)\ast (y\circ z).
\end{gather*}
\end{enumerate}

For a diagram $D$ of a virtual knot $\mathcal K$, a \emph{colouring} of $D$ with a biquandle $B$ is a map $c\colon\mathcal A(D)\to B$ which obeys the colouring rule in Fig.~\ref{fig:biquandle_operations} for any crossing of $D$.

\begin{figure}[h]
\centering\includegraphics[width=0.4\textwidth]{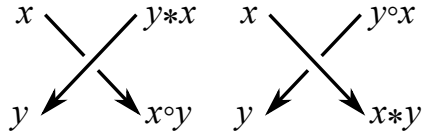}
\caption{The colouring rule}\label{fig:biquandle_operations}
\end{figure}

Let $Col_B(D)$ denote the set of colourings of the diagram $D$. For any diagram $D'$ of the same knot $\mathcal K$, there is a bijection between the colouring sets $Col_B(D)$ and $Col_B(D')$.

Given an abelian group $A$, one can associate cohomology groups $H^*(B,A)$ to a biquandle $B$~\cite{CES}. Here we need only the group $H^1(B,A)$.

A map $\theta\colon B\to A$ as called a \emph{biquandle 1-cocycle with coefficients in the group $A$} if for any $x,y\in B$ the following equality holds

\begin{equation}\label{eq:biquandle_1cocycle}
  \theta(x)-\theta(x\circ y)=\theta(y)-\theta(y\ast x).
\end{equation}

Fix a biquandle 1-cocycle $\theta\in H^1(B,A)$ with coefficients in the group $A$. For a colouring $c\in Col_B(D)$ of the diagram $D$ define a 1-chain $\delta^\theta_D=\delta^\theta_{D,c}$ by the formula

\begin{equation}\label{eq:parity_cycle_biquandle_cocycle}
  \delta^\theta_D=\sum_{a\in \mathcal A(D)} (\theta\circ c)(a)\cdot a\in C_1(D,A).
\end{equation}

By equality~\eqref{eq:biquandle_1cocycle} and the colouring rule, the chain $\delta^\theta_D$ is a 1-cycle. It is also locally invariant in the following sense.

\begin{proposition}\label{prop:biquandle_cycle_locally_invariant}
Let $f\colon D\to D'$ is an isotopy or a Reidemeister move. Let $f_*(c)\in Col_B(D')$ be the colouring that corresponds to $c\in Col_B(D)$. Then for any correspondent arcs $a\in\mathcal A(D)$ and $a'\in\mathcal A(D')$ we have $\delta^\theta_{D,c}(a)=\delta^\theta_{D',f_*(c)}(a')$.
\end{proposition}

\begin{proof}
  By definition of the bijection $f_*$, we have $c(a)=f_*(c')(a')$. Then
  \[
  \delta^\theta_{D,c}(a)=\theta(c(a))=\theta\left(f_*(c)(a')\right)=\delta^\theta_{D',f_*(c)}(a').
  \]
\end{proof}

\begin{remark}
The difference between invariance and local invariance of a 1-cycle comes from colour monodromy. We define the \emph{colour monodromy group} $Mon_B(D)$ of the diagram $D$ as the subgroup in the permutation group of the colouring set $Col_B(D)$, formed by permutations $f_*\colon Col_B(D)\to Col_B(D)$ where $f\colon D\to D$ is an arbitrary morphism (a composition of isotopies and Reidemeister moves).

If the 1-cocycle $\theta$ is invariant under the action of the colour monodromy group $Mon_B(D)$ on a given colouring $c\in Col_B(D)$:
\[
\forall f_*\in Mon_B(D)\ \forall a\in\mathcal A(D)\quad \theta(f_*(c)(a))=\theta(c(a))
\]
then the 1-cycle $\delta^\theta_{D,c}$ is an invariant cycle.

In general case, one can modify the formula~\eqref{eq:parity_cycle_biquandle_cocycle} to get an invariant 1-cycle by taking the sum over colours which belong to one orbit of the monodromy group:
\[
\tilde\delta^\theta_D=\sum_{c'\in Mon_B(D)\cdot c\subset Col_B(D)}\delta^\theta_{D,c'},
\]
or by taking the sum over all the colourings:
\[
\hat\delta^\theta_D=\sum_{c'\in Col_B(D)}\delta^\theta_{D,c'}.
\]

If the cycle $\tilde\delta^\theta_D$ (or $\hat\delta^\theta_D$) is not normalized, we can reduce it to a normalized one using Corollary~\ref{cor:cycle_normalization}.

On the other hand, one can consider the direct sum
\[
\Delta^\theta_D=\bigoplus_{c\in Col_B(D)}\delta^\theta_{D,c}\colon\V(D)\to A^{|Col_B(D)|}
\]
which is invariant under the monodromy up to permutation of the components and define an oriented parity functor with coefficients in $A^{|Col_B(D)|}$ (see Section~\ref{sect:parity_functors}).
\end{remark}


\begin{remark}
Given a biquandle 1-cocycle $\theta\in H^1(B,A)$ and a colouring $c\in Col_B(D)$, we can consider the 1-cycle $\delta^\theta_{D,c}$ and its signature $\sigma_{\theta,c}(D)=\sigma(\delta^\theta_{D,c})\in A$. The element $\sigma_{\theta,c}(D)$ is invariant under Reidemeister moves, hence, the set $\Sigma_\theta(D)=\{\sigma_{\theta,c}(D)\}_{c\in Col_B(D)}$ is an invariant of virtual knots.

In fact, the invariant $\Sigma_\theta(D)$ is closely related to biquandle cocycle invariants~\cite{CES}.
Recall that a \emph{biquandle 2-cocycle with coefficients in an abelian group $A$} is a map $\phi\colon B\times B\to A$ which obeys the following conditions:
\begin{enumerate}
\item $\phi(x,x)=0$ for all $x\in B$;
\item $\phi(x,y) -\phi(x,z)+\phi(y,z)=\phi(x\circ z,y\circ z)-\phi(x\circ y, z\ast y)+\phi(y\ast x,z\ast x)$ for all $x,y,z\in B$.
\end{enumerate}

Given a biquandle 2-cocycle $\phi$ and a colouring $c\in Col_B(D)$ one assigns to each positive crossing $v$ of the diagram (Fig.~\ref{fig:biquandle_operations} left) the \emph{Boltzman weight} $W_\phi(v,c)=\phi(x,y)$, and to each negative crossing $v$ (Fig.~\ref{fig:biquandle_operations} right) the weight $W_\phi(v,c)=-\phi(y,x)$. Then the sum
\[
\Phi_\phi(D)=\sum_{c\in Col_B(D)}\sum_{v\in\V(D)}W_\phi(v,c)\in A
\]
is an virtual knot invariant called the \emph{biquandle cycle invariant}~\cite{CES}.

Note that the term $\Phi_{\phi,c}(D)=\sum_{v\in\V(D)}W_\phi(v,c)$ is invariant under Reidemeister moves (up to identification of the colouring).

Let $B_{ind}=\Z$ be the linear biquandle with operations $x\circ y=x\ast y=x+1$. The following proposition can be checked by direct computations.

\begin{proposition}
\begin{enumerate}
\item Any 1-cocycle $\theta\in H^1(B,A)$ defines a biquandle 2-cocycle $\phi_\theta$ of the product biquandle $B\times B_{ind}$ by the formula
\begin{equation}\label{1cocycle_to_2cocycle}
\phi_\theta((x,i),(y,j))=(\theta(y)-\theta(y\ast x))\cdot(j-i);
\end{equation}
\item  $\sigma_{\theta,c}(D)=\Phi_{\phi_\theta,c}(D)$ for any $c\in Col_B(D)$.
\end{enumerate}
\end{proposition}

\end{remark}

\section{Quasi-index}\label{sect:quasiindex}

Let $\mathcal K$ be a virtual knot and $\mathfrak K$ be its diagram category. Let $A$ be an abelian group. Let $p$ be an oriented parity with coefficients in $A$ on the diagrams of the knot $\mathcal K$ and $\delta_D, D\in \mathfrak K$, be its parity cycle.

Let $D$ be a diagram of the knot $\mathcal K$. For any crossing $v$ of the diagram $D$  (see Fig.~\ref{fig:potential_parity}) consider the element (Fig.~\ref{fig:quasiindex_def})
\begin{equation}\label{eq:parity_quasiindex}
\pi_D(v)=\delta_c-\delta_a=\delta_b-\delta_d.
\end{equation}
Thus, we have a map $\pi_D\colon \V(D)\to A$.

\begin{figure}[h]
\centering\includegraphics[width=0.5\textwidth]{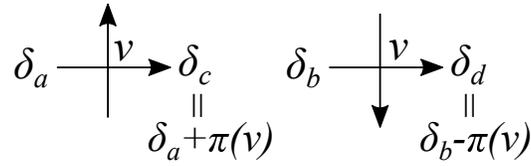}
\caption{Quasi-index of a crossing}\label{fig:quasiindex_def}
\end{figure}

\begin{definition}\label{def:quasiindex}
Let $\mathcal K$ be a virtual knot and $\mathfrak K$ be its diagram category. A family of maps $\pi_D\colon \V(D)\to A, D\in\mathfrak K$, is called a \emph{quasi-index} on the diagrams of the knot $\mathcal K$ if the following conditions hold
\begin{itemize}
\item[(Q0)] for any Reidemeister move $f\colon D\to D'$ and any crossing $v\in\V(D)$ which does not take part in the move, one has $\pi_{D}(v)=\pi_{D'}(f_*(v))$;

\item[(Q2)] $\pi_D(v_1)=\pi_D(v_2)$ for any crossings $v_1,\,v_2\in\V(D)$ to which a decreasing second Reidemeister move can be applied;

\item[(Q3)] if $v_1,v_2,v_3\in\V(D)$ are the crossings which take part in a third Reidemeister move $f\colon D\to D'$ then there exists an element $\lambda(f)\in A$ such that
\[ \pi_{D'}(f_*(v_i))=\pi_{D}(v_i)+\epsilon_\Delta(v_i)\cdot\lambda(f),\ i=1,2,3,
\]
where  $\epsilon_\Delta(v_i)$ is the incidence index of the crossing $v_i$ to the disappearing triangle $\Delta$.
\end{itemize}

A quasi-index $\pi$ is called an \emph{index} on the diagrams of the knot $\mathcal K$ if the terms $\lambda(f)$ in condition (Q3) are equal to $0$ for all third Reidemeister moves on the diagrams of $\mathcal K$. In other words, third Reidemeister moves do not change the indices of the correspondent crossings.
\end{definition}

\begin{remark}
The notion of index on knot diagrams appears naturally when one tries to give a general definition of index polynomial. This notion was introduced by M. Xu~\cite{Xu}  (under the name \emph{weak chord index}) who elaborated a more restrictive notion of \emph{chord index} given by Z. Cheng~\cite{Cheng2}, see also~\cite{CGX,CFGMX}.
\end{remark}

\begin{theorem}\label{thm:quasiindex}
Let $\mathcal K$ be a virtual knot and $p$ be an oriented parity with coefficients in a group $A$ on the diagrams of the knot $\mathcal K$. Let $\delta$ be the parity cycle of $p$ and $\pi_D\colon \V(D)\to A$, $D\in\mathfrak K$, be the map defined above. Then
\begin{enumerate}
\item the family of maps $\pi$ is a quasi-index on diagrams of the knot $\mathcal K$;

\item for any diagram $D\in\mathfrak K$ there is a unique element $\rho(D)\in A$ such that
\begin{equation}\label{eq:cycle_quasiindex}
\delta_D=\sum_{v\in\V(D)} \pi_D(v)\cdot D^r_v+\rho(D)\cdot D\in H_1(D,A).
\end{equation}
\end{enumerate}
\end{theorem}

\begin{proof}
1. Let us check the quasi-index properties.

The property (Q0) follows from the invariance of the parity cycle $\delta$.

Let $v_1$ and $v_2$ be two crossings participating in a second Reidemeister move (Fig.~\ref{fig:quasiindex_r2}). Let $a,b,c,d$ be the values of the parity cycle $\delta$ on the arc incident to $v_1,v_2$. By definition, $\pi(v_1)=a-d$ and $\pi(v_2)=a-d$. Thus, $\pi(v_1)=\pi(v_2)$. The other type of second Reidemeister move can be checked analogously.

\begin{figure}[h]
\centering\includegraphics[width=0.2\textwidth]{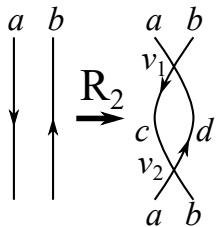}
\caption{A second Reidemeister move}\label{fig:quasiindex_r2}
\end{figure}

Let $v_1,v_2,v_3\in\V(D)$ be the crossings participating in a third Reidemeister move $f\colon D\to D'$ and $v'_1,v'_2,v'_3\in\V(D')$ be the corresponding crossings in the diagram $D'$. Then $\epsilon_\Delta(v_1)=\epsilon_\Delta(v_2)=-1$ and $\epsilon_\Delta(v_3)=1$. By definition,
\begin{align*}
\pi(v_1)&=x-a_1, & \pi(v_2)&=z-c_2, & \pi(v_3)&=a_2-x=c_1-z, \\
\pi(v'_1)&=a_2-x', & \pi(v'_2)&=c_1-z', & \pi(v'_3)&=x'-a_1=z'-c_2.
\end{align*}
Hence,
\begin{gather*}
\pi(v_1)+\pi(v_3)=a_2-a_1=\pi(v'_1)+\pi(v'_3),\\ \pi(v_2)+\pi(v_3)=c_1-c_2=\pi(v'_2)+\pi(v'_3).
\end{gather*}

Let $\lambda(f)=\pi(v'_3)-\pi(v_3)$. Then $\pi(v'_1)-\pi(v_1)=\pi(v'_2)-\pi(v_2)=-\lambda(f)$. Thus, the condition (Q3) is fulfilled.

Other cases of third Reidemeister move can be checked analogously.

\begin{figure}[h]
\centering\includegraphics[width=0.5\textwidth]{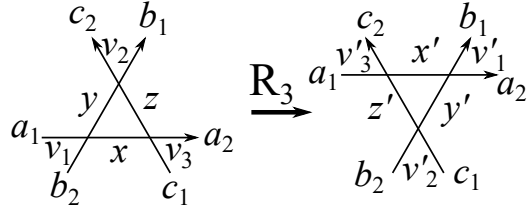}
\caption{A third Reidemeister move}\label{fig:quasiindex_r3}
\end{figure}

2. Let $D$ be a diagram of the knot $\mathcal K$, $\tilde D\subset AD(D)$ be its covering graph, and $G(D)$ be its Gauss diagram. We use a non standard orientation of the chords of $G(D)$: from the undercrossing to the overcrossing for the chords corresponding to positive crossings of $D$, and from the overcrossing to the undercrossing for the negative chords.

The parity cycle $\delta$ can be considered as a 1-cycle on the Gauss diagram $G(D)$. The labels of $\delta$ on the core circle of $G(D)$ come from the labels on the edges of $\tilde D$, i.e. on the long arcs of the diagram $D$. Then the labels of $\delta$ on the chords of $G(D)$ coincide with quasi-indices of the correspondent crossings (see Fig.~\ref{fig:quasiindex_gauss}).

\begin{figure}[h]
\centering\includegraphics[width=0.6\textwidth]{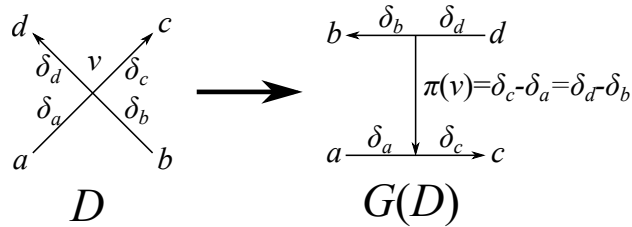}
\caption{Parity cycle on the Gauss diagram}\label{fig:quasiindex_gauss}
\end{figure}

Each (oriented) chord $v$ of $G(D)$ defines a 1-cycle which corresponds to the right half $D^r_v$ of the diagram $D$ at the crossing $v$ (Fig.~\ref{fig:gauss_right_half}).

\begin{figure}[h]
\centering\includegraphics[width=0.2\textwidth]{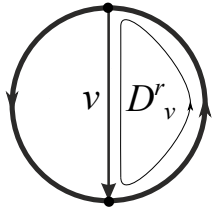}
\caption{Right half of a crossing in the Gauss diagram}\label{fig:gauss_right_half}
\end{figure}

Now, let us consider the cycle $\delta'=\delta-\sum_{v\in\V(D)}\pi(v)\cdot D^r_v\in H_1(G(D),A)$. The values of $\delta'$ on the chords are zero, then $\delta'$ is a cycle on the core circle of the Gauss diagram. Hence, the values of $\delta'$ on the arcs of the core circle are all equal to some element $\rho\in A$. Then $\delta'=\rho\cdot D$ and
\[
\delta=\sum_{v\in\V(D)}\pi(v)\cdot D^r_v+\rho\cdot D.
\]
\end{proof}

\begin{definition}
The quasi-index $\pi$ defined by the formula~\eqref{eq:parity_quasiindex} is called the \emph{parity quasi-index} of the parity $p$. The element $\rho(D)$ is called the \emph{reminder} of the parity $p$.
\end{definition}

\begin{remark}
We can give a topological interpretation of the equality~\eqref{eq:cycle_quasiindex}. Let $\bar G(D)=G(D)\cup e^2$ be the cell complex obtained from $G(D)$ by gluing a 2-cell $e^2$ to the core circle. Then $\bar G(D)$ is homotopically equivalent to the bouquet of circles $\bigvee_{v\in\V(D)}S^1$. The inclusion map $i\colon G(D)\to \bar G(D)$ induces an exact sequence in homology
\[
0\to \ker(i_*)\to H_1(G(D),A)\to H_1(\bar G(D),A)\to 0.
\]
But $H_1(\bar G(D),A)$ is isomorphic to $H_1(\bigvee_{v\in\V(D)}S^1,A)=A^{\V(D)}$, and $\ker(i_*)$ is equal to $A\cdot D\simeq A$. Thus, we have the exact sequence
\begin{equation}\label{eq:}
  0\to A\to H_1(G(D),A)\to A^{\V(D)}\to 0.
\end{equation}
The correspondence $v\mapsto D^r_v$, $v\in\V(D)$, defines a splitting of this exact sequence. Thus, $H_1(G(D),A)=A^{\V(D)}\oplus A$.
\end{remark}

The formula~\eqref{eq:cycle_quasiindex} and the intersection formula~\eqref{eq:intersection_formula} immediately imply the following
\begin{corollary}\label{quasiindex_to_parity}
Let $p$ be an oriented parity with coefficients in $A$ on diagrams of the knot $\mathcal K$, and $\pi$ be the parity quasi-index of $p$. Then for any diagram $D$ and any crossing $v\in\V(D)$
\begin{equation}\label{eq:intersection_formula_quasiindex}
p_D(v)=\sum_{v'\in\V(D)}\pi_D(v')\cdot(D^l_v\cdot D^r_{v'})-\rho(D)\cdot ip_D(v).
\end{equation}
\end{corollary}

Thus, the parity of a crossing is defined by the parity quasi-index and the reminder term $\rho(D)\in A$. Let us look how the reminder changes under Reidemeister moves.

\begin{proposition}\label{prop:reminder_reidemeister}
Let $p$ be an oriented parity with coefficients in $A$ on diagrams of the knot $\mathcal K$, and $\rho$ be the reminder of $p$.
\begin{enumerate}
\item Let $f\colon D\to D'$ be an increasing first Reidemeister move, and $v'_0\in\V(D')$ be the new crossing. Then
\[
\rho(D')=\rho(D)-\pi_{D'}(v'_0)\cdot k(f)
\]
where the coefficient $k(f)$ depends on the type of the appeared loop (Fig.~\ref{fig:loop_types}):
\[
k(f)=\left\{\begin{array}{cl}
              1, & \mbox{the loop is of type } l_\pm,  \\
              0, & \mbox{the loop is of type } r_\pm.
            \end{array}\right.
\]

\begin{figure}[h]
\centering\includegraphics[width=0.4\textwidth]{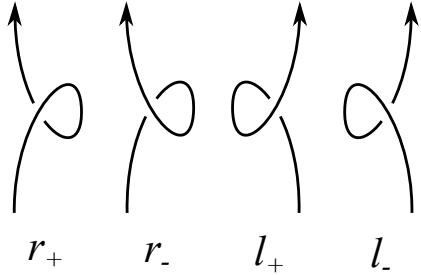}
\caption{Types of loops}\label{fig:loop_types}
\end{figure}

\item Let $f\colon D\to D'$ be an increasing second Reidemeister move, and $v'_1,v'_2\in\V(D')$ be the new crossings.
Then
\[
\rho(D')=\rho(D)-\pi_{D'}(v'_1).
\]
\item Let $f\colon D\to D'$ be a third Reidemeister move applied to crossings $v_1,v_2,v_3\in\V(D)$, and $v'_1,v'_2,v'_3$ be the corresponding crossings in $D'$.
Then
\[
\rho(D')=\rho(D)-\lambda(f)\cdot k(f).
\]
where $\lambda(f)$ is the term in property (Q3) and the coefficient $k(f)$ is determined by the type of the third Reidemeister move (Fig.~\ref{fig:r3_types}).

\begin{figure}[h]
\centering\includegraphics[width=0.55\textwidth]{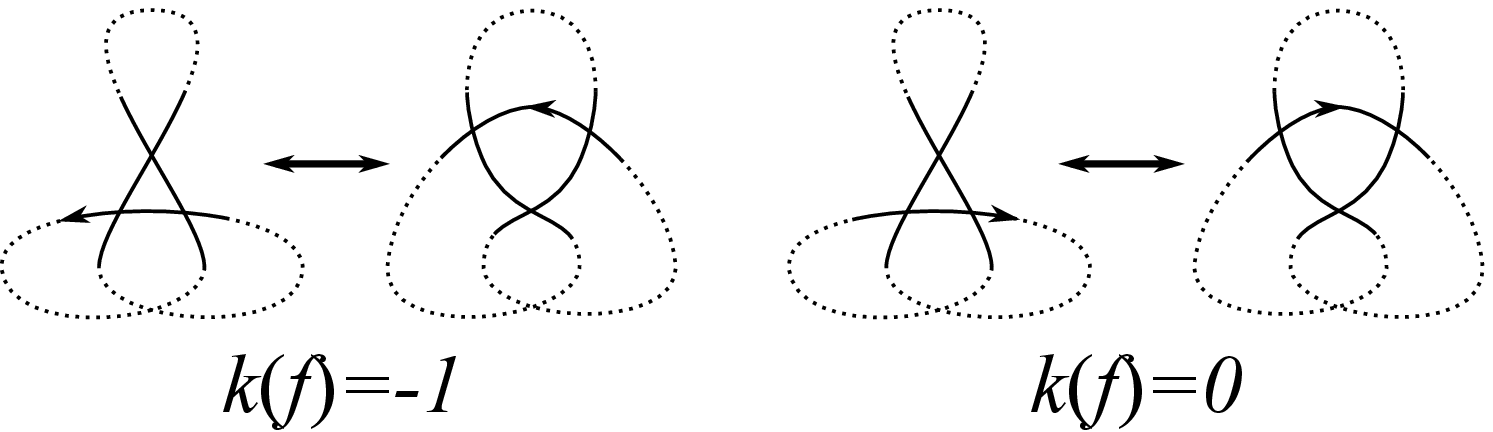}\\ \includegraphics[width=0.6\textwidth]{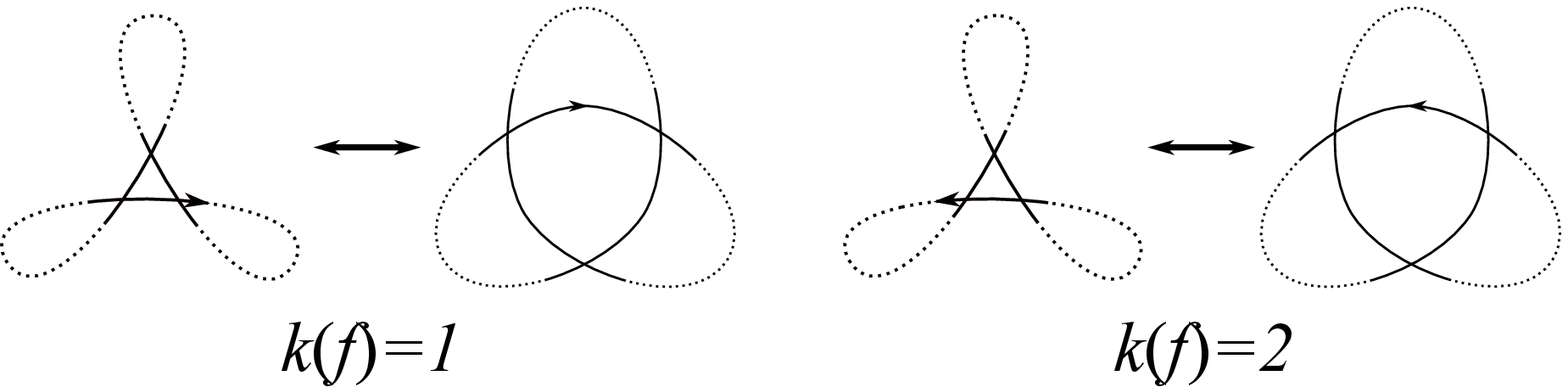}
\caption{Types of third Reidemeister moves}\label{fig:r3_types}
\end{figure}

\end{enumerate}
\end{proposition}

\begin{proof}
  The change of the reminder $\rho$ under a Reidemeister move $f\colon D\to D'$ is determined by the invariance condition of the parity cycle $\delta$.

  The diagrams $D$ and $D'$ can be lifted to the abstract knot diagram $AD(D')$, and their lifts differ only in the region where the Reidemeister move occurs. Just like in Theorem~\ref{thm:parity_cycle_inverse}, consider a map $h\colon AD(D')\to AD(D')$ which contracts this region to a point, and let $\hat D=h(D)=h(D')$. The map $h$ induces maps of 1-chains from $C_1(D,A)$ and $C_1(D',A)$ to $C_1(\hat D,A)$ which we denote by $f_*$. Then $h_*(D)=h_*(D')=\hat D$, $h_*(D^r_v)=h_*(D'^r_{f_*(v)})=\hat D^r_v$ for all crossings $v\in\V(D)$, and $h_*(\delta_D)=h_*(\delta_{D'})$ by the invariance of $\delta$.

1. Let $f\colon D\to D'$ be an increasing first Reidemeister move with the new crossing $v'_0\in\V(D')$. Then
\begin{multline*}
  0=h_*(\delta_{D'})-h_*(\delta_D)=\\
  \sum_{v'\in\V(D')} \pi_{D'}(v')\cdot h_*(D'^r_{v'})+\rho(D')\cdot h_*(D')-\sum_{v\in\V(D)} \pi_D(v)\cdot h_*(D^r_v)-\rho(D)\cdot h_*(D)=\\
  \sum_{v\in\V(D)} (\pi_{D'}(f_*(v))-\pi_D(v))\cdot \hat D^r_v+\pi_{D'}(v'_0)\cdot h_*(D'^r_{v'_0}) +(\rho(D')-\rho(D))\cdot \hat D=\\
  \pi_{D'}(v'_0)\cdot h_*(D'^r_{v'_0}) +(\rho(D')-\rho(D))\cdot \hat D
\end{multline*}
But $h_*(D'^r_{v'_0})=h_*(D')=\hat D$ for loops of type $l_\pm$ and $h_*(D'^r_{v'_0})=0$ for loops of type $r_\pm$. Then $(\rho(D')-\rho(D)+\pi_{D'}(v'_0)\cdot k(f))\cdot\hat D=0\in C_1(\hat D,A)$. Hence, $\rho(D')-\rho(D)+\pi_{D'}(v'_0)\cdot k(f)=0$.

2. Let $f\colon D\to D'$ be an increasing second Reidemeister move, and $v'_1,v'_2\in\V(D')$ be the new crossings. Then by properties (Q0) and (Q2) we have
\begin{multline*}
0=h_*(\delta_{D'})-h_*(\delta_D)=\\
\pi_{D'}(v'_1)\cdot h_*(D'^r_{v'_1})+\pi_{D'}(v'_2)\cdot h_*(D'^r_{v'_2})+\rho(D')\cdot h_*(D')-\rho(D)\cdot h_*(D)=\\
\pi_{D'}(v'_1)\cdot h_*(D'^r_{v'_1}+D'^r_{v'_2})+(\rho(D')-\rho(D))\cdot \hat D.
\end{multline*}
But $h_*(D'^r_{v'_1}+D'^r_{v'_2})=h_*(D')=\hat D$, hence, $\rho(D')=\rho(D)-\pi_{D'}(v'_1)$.

3. Let $f\colon D\to D'$ be a third Reidemeister move applied to crossings $v_1,v_2,v_3\in\V(D)$, and $v'_1,v'_2,v'_3$ be the corresponding crossings in $D'$. Since $(D')^r_{v'_i}=D^r_{v_i}\in H_1(AD(D),A)$ and $D'=D$, we have
\begin{multline*}
  0 = h_*(\delta_{D'})-h_*(\delta_D) = \sum_{i=1}^3(\pi_{D'}(v'_i)-\pi_D(v_i))\cdot \hat D^r_{v_i}+(\rho(D')-\rho(D))\cdot \hat D=\\
  \sum_{i=1}^3 \lambda(f)\cdot\epsilon_\Delta(v_i)\cdot \hat D^r_{v_i}+(\rho(D')-\rho(D))\cdot \hat D=\\ \lambda(f)\cdot\sum_{i=1}^3\epsilon_\Delta(v_i)\cdot \hat D^r_{v_i}+(\rho(D')-\rho(D))\cdot \hat D.
\end{multline*}
The sum $\sum_{i=1}^3\epsilon_\Delta(v_i)\cdot \hat D^r_{v_i}=h_*\left(\sum_{i=1}^3\epsilon_\Delta(v_i)\cdot D^r_{v_i}\right)$ is equal to $h_*(k(f)\cdot D)=k(f)\cdot \hat D$ (see Fig.~\ref{fig:r3_types}). Hence, $\rho(D')=\rho(D)-\lambda(f)\cdot k(f)$.
\end{proof}

\begin{remark}
Proposition~\ref{prop:reminder_reidemeister} means that the reminder $\rho$ and hence the parity cycle $\delta$ are determined by the quasi-index $\pi$ up to constant summand $c\cdot D$. We can eliminate the ambiguity if we set the value $\rho(D)$ for an arbitrary diagram $D$.
\end{remark}

Let us give an example of a quasi-index which is not an index.

\begin{example}\label{exa:quasiindex_example}
Consider the biquandle $B=\{1,2,3\}$ with the operations given by the matrices
\[
\circ=\left(\begin{array}{ccc}
              1 & 1 & 1 \\
              3 & 3 & 2 \\
              2 & 2 & 2
            \end{array}\right),\qquad
\ast=\left(\begin{array}{ccc}
              1 & 2 & 3 \\
              2 & 3 & 1 \\
              3 & 1 & 2
            \end{array}\right)
\]
and the 1-cocycle $\theta\in H^1(B,\Z_3)$ such that $\theta(1)=0$, $\theta(2)=1$ and $\theta(3)=-1$. The biquandle $B$ is isomorphic to the biquandle $\Z_3$ with $x\circ y = -x$, $x\ast y= x+y$ and $\theta(x)=x$ for $x,y\in\Z_3$.

Consider the coloring of the unknot diagram $D$ with $B$ given in Fig.~\ref{fig:quasiindex_example} left. Let $\delta=\delta^\theta_{D,c}$ be the corresponding locally invariant cycle. Since $D$ is an unknot diagram, the cycle $\delta$ is normalized. Let $\pi$ be the quasi-index of the cycle $\delta$. Then the quasi-indices of the crossings participating in the third Reidemeister move $f$ in Fig.~\ref{fig:quasiindex_example} change from $\theta(1)-\theta(2)=2$, $\theta(2)-\theta(3)=2$, $\theta(1)-\theta(2)=2$ to $\theta(1)-\theta(3)=1$, $\theta(1)-\theta(1)=0$, $\theta(1)-\theta(3)=1$. Then $\lambda(f)=1\ne 0$. Thus, $\pi$ is not an index.

\begin{figure}[h]
\centering\includegraphics[width=0.5\textwidth]{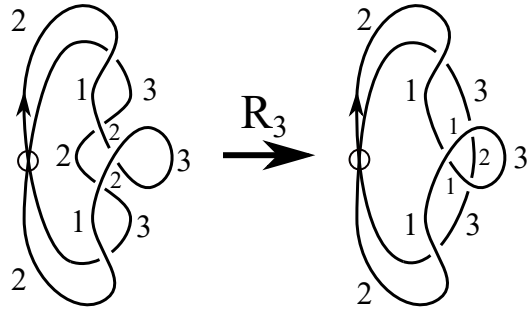}
\caption{Colouring changes by a third Reidemeister move}\label{fig:quasiindex_example}
\end{figure}

Note that this example is not quite correct because $\delta$ is not invariant due to the colouring monodromy and does not define a parity. Figure~\ref{fig:colouring_monodromy} implies that one can interchange the colours $2$ and $3$ in any diagram of the unknot but the cocycle $\theta$ is not invariant under this permutation.

\begin{figure}[h]
\centering\includegraphics[width=0.5\textwidth]{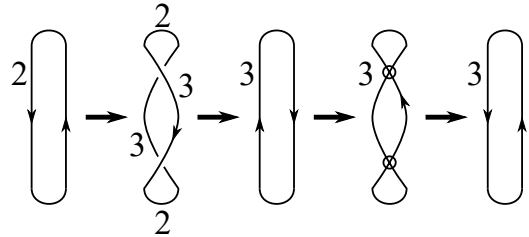}
\caption{Colouring monodromy}\label{fig:colouring_monodromy}
\end{figure}

The biquandle $B$ with the cocycle $\theta$ can define a parity with an invariant cycle whose quasi-index is not an index, for a knot $\mathcal K$ which has a colouring with $B$ such that it has an edge coloured with $2$ (or $3$) and its orbit under the colouring monodromy group $Mon_B(\mathcal K)$ is trivial: for such a knot we take its connected sum with the unknot diagram in Fig.~\ref{fig:quasiindex_example}.  Such a knot could be found among irreducible odd knots in the sense of~\cite{M2} but the author do not know an explicit example.
\end{example}

\begin{remark}
Let $\mathcal K$ be a virtual knot, $B$ be a biquandle and $\theta\in H^1(B,A)$ be a cocycle which defines an oriented parity $p$ with coefficients in $A$ on the diagrams of the knot $\mathcal K$. Then the following conditions ensure that the quasi-index $\pi$ of the parity $p$ is an index on the diagrams of $\mathcal K$: for all $x,y,z\in B$
\begin{align}\label{eq:index_conditions}
\theta(x)-\theta(x\circ y)-\theta(x\circ z)+\theta\left((x\circ y)\circ(z\circ y)\right)=0,\nonumber\\
\theta(x)-\theta(x\circ y)-\theta(x\ast z)+\theta\left((x\circ y)\ast(z\circ y)\right)=0,\\
\theta(x)-\theta(x\ast y)-\theta(x\ast z)+\theta\left((x\ast y)\ast(z\ast y)\right)=0.\nonumber
\end{align}
The conditions above come from the equations $\pi_{D'}(v'_i)-\pi_D(v_i)=0$ for the crossings $v_i$ participating in third Reidemeister moves.
\end{remark}

Now, let $\pi$ be an arbitrary quasi-index on the diagrams of a virtual knot $\mathcal K$ with coefficients in an abelian group $A$. Fix a diagram $D$ of the knot $\mathcal K$ and an element $\rho(D)\in A$. Using the formulas of Proposition~\ref{prop:reminder_reidemeister}, we can extend the reminder $\rho$ to other diagrams of $\mathcal K$, and define the cycle $\delta^\pi=\delta$ by the formula~\eqref{eq:cycle_quasiindex}. There are two obstacles for $\delta^\pi$ to be a normalized invariant cycle and define a parity. The first one is the monodromy.

Let $f\colon D\to D$ be a diagram morphism, i.e. a sequence
\[
D\stackrel{f_1}{\rightarrow}D_1 \stackrel{f_2}{\rightarrow}\cdots\stackrel{f_{n-1}}{\rightarrow}D_{n-1}\stackrel{f_n}{\rightarrow}D
\]
 of Reidemeister moves $f_1,\dots,f_n$. Define the \emph{monodromy} $\Delta_\pi(f)$ of the morphism $f$ as the sum
\begin{equation*}
  \Delta_\pi(f)=\sum_{i=1}^n \Delta_\pi(f_i)
\end{equation*}
where $\Delta_\pi(f_i)=\rho(D_{i+1})-\rho(D_i)$ is determined by Proposition~\ref{prop:reminder_reidemeister}. The set
\[
Mon(\pi)=\left\{ \Delta_\pi(f) \mid f\colon D\to D \right\}
\]
is called the \emph{monodromy group of the quasi-index} $\pi$. Note that $Mon(\pi)$ is a subgroup in $A$ and it does not depend on the choice of the diagram $D$ of the knot $\mathcal K$.

If the monodromy group is not zero then the cycle $\delta^\pi$ is ambiguous because Reidemeister moves can shift the reminder $\rho(D)$ by any element of $Mon(\pi)$.

\begin{example}
Consider the constant quasi-index $\pi\equiv 1$ with coefficients in $\Z$ on the unknot diagrams. Take the diagram $D$ in Fig.~\ref{fig:quasiindex_monodromy} and set $\rho(D)=0$. Consider the morphism $f\colon D\to D$ which consists of two first, one second Reidemeister move and detour moves (Fig.~\ref{fig:quasiindex_monodromy}). Then $\Delta_\pi(f)=1$. Hence, the monodromy group $Mon(\pi)$ coincides with the coefficient group $\Z$.

\begin{figure}[h]
\centering\includegraphics[width=0.5\textwidth]{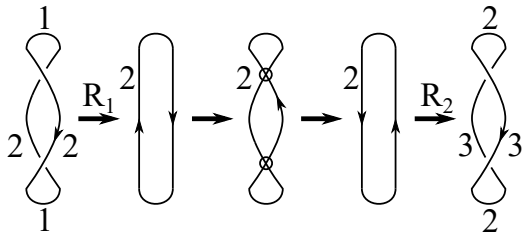}
\caption{Quasi-index monodromy}\label{fig:quasiindex_monodromy}
\end{figure}
\end{example}

The other obstacle for $\delta^\pi$ is the signature of the quasi-index.

\begin{definition}\label{def:quasiindex_signature}
Let $\pi$ be an quasi-index on the diagrams of a virtual knot $\mathcal K$ with coefficients in an abelian group $A$. Let $D$ be a diagram of $\mathcal K$. Then the \emph{signature $\sigma(\pi)$ of the quasi-index $\pi$} is defined by the formula
\begin{equation}\label{eq:quasiindex_signature}
  \sigma(\pi)= \sum_{v\in\V(D)}\pi_D(v)\cdot(D\cdot D^r_{v}) = -\sum_{v\in\V(D)}\pi_D(v)\cdot ip_D(v)\in A.
\end{equation}
Note that $\sigma(\pi)=D\cdot\delta_D$ where $\delta$ is defined by formula~\eqref{eq:cycle_quasiindex} with any choice of $\rho(D)$.
\end{definition}

\begin{proposition}\label{prop:quasiindex_signature_invariance}
Let $\pi$ be an quasi-index on the diagrams of a virtual knot $\mathcal K$ with coefficients in an abelian group $A$. Then the signature $\sigma(\pi)$ is invariant under Reidemeister moves.
\end{proposition}

\begin{proof}
  Let $f\colon D\to D'$ be a Reidemeister move between two diagrams of the knot $\mathcal K$. Let  $\sigma_D(\pi)$ and $\sigma_{D'}(\pi)$ be the signatures calculated for the diagrams $D$ and $D'$ by formula~\eqref{eq:quasiindex_signature}. Lift the diagrams $D$ and $D'$ to subsgraphs in the abstract knot diagram $AD(D)$ (or $AD(D')$). By Proposition~\ref{prop:reminder_reidemeister}
  \[
  \sum_{v\in\V(D)}\pi_D(v)\cdot D^r_{v}-\sum_{v'\in\V(D')}\pi_{D'}(v')\cdot D'^r_{v'}=c(f)\cdot D\in H_1(AD(D),A)
  \]
for some $c(f)\in A$. Then $\sigma_{D'}(\pi)-\sigma_{D}(\pi)=c(f)\cdot (D\cdot D)=0$
\end{proof}

Thus, if the signature $\sigma(\pi)$ of the quasi-index $\pi$ is not zero then the cycle $\delta^\pi$ is not normalized. But we make it normalized like in Corollary~\ref{cor:cycle_normalization}. Thus, we have

\begin{theorem}\label{thm:quasiindex_parity}
Let $\pi$ be an quasi-index on the diagrams of a virtual knot $\mathcal K$ with coefficients in an abelian group $A$. Let $\sigma(\pi)$ and $Mon(\pi)$ be the signature and the monodromy group of $\pi$. Consider the factor-group
\[
\bar A = A/(Mon(\pi)+\left<\sigma(\pi)\right>).
\]
 Then the formula
\begin{equation}\label{eq:quasiindex_parity}
p^\pi_D(v)=\sum_{v'\in\V(D)}\pi_D(v')\cdot(D^l_v\cdot D^r_{v'}).
\end{equation}
defines an oriented parity $p^\pi$ with coefficients in $\bar A$ on the diagrams of the knot $\mathcal K$.
\end{theorem}

On the other hand, any quasi-index $\pi$ defines an oriented \emph{parity functor} with coefficients in $A\oplus\Z$ (see Theorem~\ref{thm:quasiindex_parity_functor}).

We postpone the study of quasi-indices to further papers and concentrate below only on indices.

\section{From index to parity cycle}\label{sect:index_to_parity}

By Definition~\ref{def:quasiindex}, we can define an index $\pi$ on the diagrams of of a virtual knot $\K$ with coefficients in an abelian group $A$ as a family of maps $\pi_D\colon\V(D)\to A$, $D\in\mathfrak K$, such that
\begin{itemize}
\item[(I0)] for any Reidemeister move $f\colon D\to D'$ and any crossings $v\in\V(D)$ and $v'\in\V(D')$ such that $v'=f_*(v)$, one has $\pi_{D}(v)=\pi_{D'}(v')$;
\item[(I2)] $\pi_D(v_1)=\pi_D(v_2)$ for any crossings $v_1,\,v_2\in\V(D)$ to which a decreasing second Reidemeister move can be applied.
\end{itemize}

As in the previous section, our goal is to construct an invariant 1-cycle $\delta$ and a parity for a given index $\pi$.

Let  $\pi$ be an index on the diagrams of of a virtual knot $\K$ with coefficients in an abelian group $A$. Let $D$ be a diagram of the knot $\K$. Define a $1$-cycle $\tilde\delta^\pi_D\in H_1(D,A)$ by the formula
\begin{equation}\label{eq:quasiindex_cycle_original}
    \tilde\delta^\pi_D=\sum_{v\in\V(D)}\pi_D(v)\cdot D^r_v.
\end{equation}

By property (I0), the 1-cycle $\delta^\pi$ is invariant under third Reidemeister moves.

Let $f\colon D\to D'$ be an increasing second Reidemeister move and $v'_+,v'_-\in\V(D')$, $sgn(v'_\pm)=\pm 1$, be the new crossings. By property (I2), $\pi_{D'}(v'_+)=\pi_{D'}(v'_-)$, and Proposition~\ref{prop:reminder_reidemeister} implies that
\[
\tilde\delta^\pi_{D'}-\tilde\delta_D^\pi=\pi_{D'}(v'_-)\cdot D
\]
in $H_1(AD(D'),A)$. In order to compensate the increment of the cycle $\tilde\delta^\pi$, we modify the formula~\eqref{eq:quasiindex_cycle_original}.

Let $\V(D)_-$ and $\V(D)_+$ be the sets of negative and positive crossing in $D$.  For any crossing $v\in\V(D)$ define its \emph{signed halves} $D^\pm_v$ by the formulas $D^l_c=D^{sign(v)}_v$, $D^r_c=D^{-sign(v)}_v$.

Consider the 1-cycle
\begin{multline*}
\delta^\pi_D=\tilde\delta^\pi_D-\sum_{v\in\V(D)_-}\pi_D(v)\cdot D =\\
 \sum_{v\in\V(D)_+}\pi_D(v)\cdot D^r_v+\sum_{v\in\V(D)_-}\pi_D(v)\cdot (D^r_v-D)=\\
\sum_{v\in\V(D)_+}\pi_D(v)\cdot D^r_v-\sum_{v\in\V(D)_-}\pi_D(v)\cdot D^l_v.
\end{multline*}
Thus,
\begin{equation}\label{eq:quasiindex_cycle_base}
\delta^\pi_D=\sum_{v\in\V(D)} sgn(v)\pi_D(v)\cdot D^-_v.
\end{equation}

By definition of $\delta^\pi$, the following proposition holds.

\begin{proposition}\label{prop:index_delta_r23_invariance}
 The cycle $\delta^\pi$ is invariant under second and third Reidemeister moves.
\end{proposition}

\begin{remark}
We can consider the cycle $\delta'^\pi_D=\tilde\delta^\pi-\sum_{v\in\V(D)_+}\pi_D(v)\cdot D$ which is also invariant under second and third Reidemeister moves. The difference between $\delta^\pi$ and $\delta'^\pi$
\[
\delta^\pi-\delta'^\pi=\left(\sum_{v\in\V(D)_+}\pi_D(v)-\sum_{v\in\V(D)_-}\pi_D(v)\right)\cdot D= lk(\pi)\cdot D
\]
is expressed using the linking invariant $lk(\pi)$ of the index $\pi$ (see Section~\ref{sect:derived_parity}).
\end{remark}

Let us consider an increasing first Reidemeister move $f\colon D\to D'$ and let $v'\in\V(D')$ be the new crossing. There can be four types of the crossing $v'$, see Fig.~\ref{fig:loop_types}. By property (I0), there are elements $\pi^{l-}, \pi^{l+}, \pi^{r-}, \pi^{r+}\in A$ which are the index values of the corresponding types of loop crossings. Due to Whitney's trick and property (I2), $\pi^{l-}=\pi^{r+}=\pi^\bullet$ and $\pi^{l+}=\pi^{r-}=\pi^\circ$. Thus, the index $\pi_{D'}(v')$ coincides with one of two values $\pi^\bullet$ or $\pi^\circ$.

The increment of the 1-cycle $\delta^\pi$ for the move $f$ is equal to
\[
\delta^\pi_{D'}-\delta_D^\pi=k\cdot\pi_{D'}(v')\cdot D
\]
where $k=-1$ if $v'$ is of type $r_-$, $k=0$ for $r_+$ and $l_-$, and $k=1$ for the type $l_+$. Hence,
\begin{equation}\label{eq:index_delta_r1}
  \delta^\pi_{D'}-\delta_D^\pi=\left\{\begin{array}{cl}
                                        sgn(v')\pi^\circ\cdot D, & v'\mbox{ is of type $r_-$ or $l_+$}, \\
                                        0, & v'\mbox{ is of type $r_+$ or $l_-$}.
                                      \end{array}\right.
\end{equation}

We can consider several strategies to deal with first Reidemeister moves (some of them fit to quasi-index case):
\begin{itemize}
\item reduced index (Section~\ref{subsect:reduced_index}),
\item almost classical knots (Section~\ref{subsect:almost_classical_knots}),
\item regular knots (Section~\ref{subsect:regular_knots}),
\item rotational virtual knots (Section~\ref{subsect:rotational_knots}),
\item long knots (Section~\ref{sect:long_knots}),
\item parity functors (Section~\ref{sect:parity_functors}).
\end{itemize}

\subsection{Reduced index}\label{subsect:reduced_index}

By equation~\eqref{eq:index_delta_r1} the invariance under first Reidemeister move takes place automatically if we are lucky enough to have $\pi^\circ=0$.

\begin{definition}\label{def:index_reduced}
An index $\pi$ on diagrams of a virtual knot $\K$ with coefficients in an abelian group $A$ is \emph{$l_+$-reduced} if the index value of the crossings of types $l_+$ and $r_-$ (Fig.~\ref{fig:loop_types}) is zero: $\pi^\circ=0$.

Analogously, the index $\pi$ is \emph{$r_+$-reduced} if the index values of $l_-$- and $r_+$-crossings is zero: $\pi^\bullet=0$. And $\pi$ is \emph{R1-reduced} if $\pi^\circ=\pi^\bullet=0$.
\end{definition}

\begin{theorem}\label{thm:index_parity_reduced}
  Let $\pi$ be an $l_+$-reduced index on diagrams of a virtual knot $\K$ with coefficients in an abelian group $A$. Then the formula~\eqref{eq:quasiindex_cycle_base} defines an invariant 1-cycle $\delta^\pi_D$ and the formula
\begin{equation}\label{eq:index_parity_reduced}
p^\pi_D(v)=\sum_{v'\in\V(D)}sgn(v')\pi_D(v')\cdot(D^l_v\cdot D^-_{v'})
\end{equation}
defines an oriented parity $p^\pi$ on diagrams of the knot $\K$ with coefficients in the group $\bar A=A/\left<\sigma(\pi)\right>$, where $\sigma(\pi)\in A$ is the signature of the index $\pi$.
\end{theorem}

For a general index $\pi$, we can either suppress or evade the index value $\pi^\circ$. The following statement (cf.~Theorem~\ref{thm:quasiindex_parity}) represents the first option.
\begin{corollary}
Let $\pi$ be an index on diagrams of a virtual knot $\K$ with coefficients in an abelian group $A$ and let $\pi^\circ\in A$ be the index value of the crossings of types $l_+$ and $r_-$. Then the formula~\eqref{eq:quasiindex_cycle_base} defines an invariant 1-cycle $\delta^\pi_D$ and the formula~\eqref{eq:index_parity_reduced} defines an oriented parity $p^\pi$ on diagrams of the knot $\K$ with coefficients in the group $\bar A=A/\left<\pi^\circ,\sigma(\pi)\right>$, where $\sigma(\pi)$ is the signature of the index $\pi$.
\end{corollary}

\begin{definition}\label{def:index_lplus_reduction}
Given an index $\pi$ on diagrams of a virtual knot $\K$ with coefficients in an abelian group $A$, define its \emph{$l_+$-reduction} as the a family of maps $\bar\pi_D\colon\V(D)\to A$ ,$D\in\mathfrak K$, set by the formula
\begin{equation*}
  \bar\pi_D(v)=\left\{\begin{array}{cl}
                        \pi_D(v), & \pi_D(v)\ne\pi^\circ, \\
                        0, & \pi_D(v)=\pi^\circ.
                      \end{array}\right.
\end{equation*}
Here $\pi^\circ\in A$ is the index value of the crossings of types $l_+$ and $r_-$.
\end{definition}

The following statement is evident.
\begin{proposition}
The $l_+$-reduction $\bar\pi$ is an $l_+$-reduced index on diagrams of $\K$ with coefficients in $A$.
\end{proposition}

Hence, Theorem~\ref{thm:index_parity_reduced} holds for $\bar\pi$. We can reformulate it as follows.

\begin{corollary}\label{cor:index_parity_reduced}
Let $\pi$ be an index on diagrams of a virtual knot $\K$ with coefficients in an abelian group $A$ and let $\pi^\circ\in A$ be the index value of the crossings of types $l_+$ and $r_-$. Then the formula
\begin{equation}\label{eq:index_reduction_delta}
\bar\delta^\pi_D=\sum_{v\in\V(D)\colon \pi_D(v)\ne\pi^\circ} sgn(v)\pi_D(v)\cdot D^-_v
\end{equation}
 defines an invariant 1-cycle $\bar\delta^\pi_D$ and the formula
\begin{equation}\label{eq:index_reduction_parity}
\bar p^\pi_D(v)=\sum_{v'\in\V(D)\colon \pi_D(v')\ne\pi^\circ}sgn(v')\pi_D(v')\cdot(D^l_v\cdot D^-_{v'})
\end{equation}
defines an oriented parity $\bar p^\pi$ on diagrams of the knot $\K$ with coefficients in the group $\bar A=A/\left<\bar\sigma(\pi)\right>$, where
$\bar\sigma(\pi)=\sigma(\bar\pi)$ is the signature of the $l_+$-reduction $\bar\pi$.
\end{corollary}

One can consider the reduced 1-cycle $\bar p^\pi$ is a distant analogue of odd writhe polynomial~\cite{Cheng}.

\begin{remark}
Analogous results take place when the index $\pi$ is \emph{$r_+$-reduced}, i.e. $\pi^\bullet=0$.
\end{remark}

We will return to the case of reduced indices in Section~\ref{sect:derived_parity}.

\subsection{Almost classical knots}\label{subsect:almost_classical_knots}

Let us recall the definition of almost classical knot~\cite{BGHNW}.

\begin{definition}\label{def:almost_classical_knot}
A virtual knot diagram $D$ is called \emph{almost classical} if for any crossing $v\in\V(D)$ its index parity is zero: $ip_D(v)=0$.
Equivalently, a diagram $D$ is almost classical iff $D=0\in H_1(AD(D),\Z)$.

A virtual knot $\K$ is {\em almost classical} if it has an almost classical diagram.
\end{definition}

The category of almost classical diagrams embeds in the category of virtual knot diagrams and Reidemeister moves between them in the following sense: if $D$ and $D'$ are almost classical diagrams of one knot $\K$ then there exist a sequence of Reidemeister moves $D\to D_1\to\cdots\to D_n\to D'$ such that all the intermediate diagrams $D_1,\dots,D_n$ are almost classical. This fact can be proved by using~\emph{parity projection}~\cite{IMN1}. Denote the set of almost classical diagrams of $\K$ by $\mathfrak K_{ac}$.

The condition $D=0$ in the definition of almost classical knots means we need not to worry about the invariance of the cycle $\delta^pi$. Theorem~\ref{thm:quasiindex_parity} and definition of almost classical diagrams imply the following result.

\begin{theorem}\label{thm:quasiindex_parity_almost_classical}
  Let $\K$ be an almost classical knots and $\pi$ be a quasi-index on {\em all} diagrams of $\K$ with coefficients in an abelian group $A$. Then the formula~\eqref{eq:quasiindex_cycle_original} defines an invariant normalized 1-cycle $\delta^\pi_D$ and the formula~\eqref{eq:quasiindex_parity} defines an oriented parity $p^\pi_D$, $D\in\mathfrak K_{ac}$, on almost classical diagrams of $\K$ with coefficients in the group $A$.
\end{theorem}

\begin{example}
Let $\K$ be the virtual knot 5.2012 (Fig.~\ref{fig:knot_5.2012}). It is almost classical~\cite{BGHNW}. Consider the reduced stable parity on the knot~\cite{N3}. The based matrix $T$ of the diagram $D$ in Fig.~\ref{fig:knot_5.2012} is primitive. Each crossing of $D$ defines a nonzero primitive tribe. Taking into account the signs of crossings, there are no nontrivial symmetries of the diagram $D$ and the based matrix $T$. Then we can define an index $\pi$ with coefficients in $A=\Z^5$ such that the crossings of $D$ correspond to the basis of $A$. Then the parities $p^\pi_D(v)$, $v\in\V(D)$, correspond to the rows of the based matrix $T$ (without the first row and the first column).

\begin{figure}[h]
\centering
\raisebox{-0.7\height}{\includegraphics[width=0.25\textwidth]{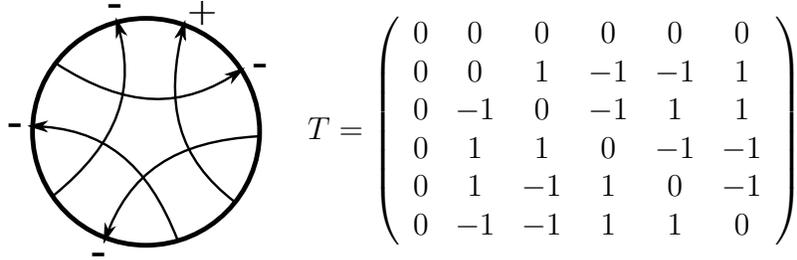}} \quad
\raisebox{-0.5\height}{$T=\left(
\begin{array}{cccccc}
 0 & 0 & 0 & 0 & 0 & 0 \\
 0 & 0 & 1 & -1 & -1 & 1 \\
 0 & -1 & 0 & -1 & 1 & 1 \\
 0 & 1 & 1 & 0 & -1 & -1 \\
 0 & 1 & -1 & 1 & 0 & -1 \\
 0 & -1 & -1 & 1 & 1 & 0 \\
\end{array}
\right)$}
\caption{Gauss diagram and the based matrix of the knot 5.2012}\label{fig:knot_5.2012}
\end{figure}

\end{example}

\subsection{Regular virtual knots}\label{subsect:regular_knots}

Let us consider virtual knot diagrams up to regular isotopy of classical crossings, i.e. up to all moves except R1.

\begin{definition}\label{def:regular_virtual_knots}
A {\em regular virtual knot} is an equivalence class of virtual knot diagrams modulo isotopies, second and third Reidemeister moves and detour moves.
\end{definition}

Thus, we do not have first Reidemeister moves and do not need to check invariance under them. Note also, that the statements of Sections~\ref{sect:potentials} and~\ref{sect:quasiindex} can be proved without first Reidemeister move. Then we can formulate the theorem.

\begin{theorem}\label{thm:quasiindex_parity_regular}
  Let $\K$ be a regular virtual knot and $\pi$ be an index on diagrams of $\K$ with coefficients in an abelian group $A$. Then the formula~\eqref{eq:quasiindex_cycle_base} defines an invariant 1-cycle $\delta^\pi_D$ and the formula~\eqref{eq:index_parity_reduced}
  defines an oriented parity $p^\pi_D$, $D\in\mathfrak K_{ac}$, on the regular knot $\K$ with coefficients in the group $\bar A=A/\left<\sigma(\pi)\right>$, where $\sigma(\pi)\in A$ is the signature of the index $\pi$.
\end{theorem}

\subsection{Rotational virtual knots}\label{subsect:rotational_knots}

An analogous result can be obtained if one forbids first virtual Reidemeister moves (instead of first classical Reidemeister moves).

\begin{definition}\label{def:rotational_virtual_knots}
A {\em rotational virtual knot}~\cite{K3} is an equivalence class of virtual knot diagrams modulo isotopies, classical Reidemeister moves and regular detour moves (i.e. moves which includes virtual analogues of second and third Reidemeister move but not first Reidemeister move).
\end{definition}

Let $D$ be a diagram of a rotational virtual knot $\K$. Then $D$ can be considered as plane diagram. In rotational knot theory, the {\em writhe} $w(D)=\sum_{v\in\V(D)}sgn(v)$ and the {\em rotation number} $\tau(D)$ of the diagram can be changed only by a first classical Reidemeister move.

Let $f\colon D\to D'$ be an increasing first Reidemeister move and $v'\in\V(D)$ be the new crossing. Then we have the following table.

\begin{center}
\begin{tabular}{l*{4}{c}}
  Type & $l_+$ & $l_-$ & $r_+$ & $r_-$ \\
  \hline
  $\Delta w$ & $1$ & $-1$ & $1$ & $-1$ \\
  $\Delta\tau$ & $1$ & $1$ & $-1$ & $-1$ \\
  $\Delta\delta^\pi$ & $\pi^\circ D$ & $0$ & $0$ & $-\pi^\circ D$ \\
  \hline
\end{tabular}
\end{center}

The table shows that the writhe and the rotation number can be used for correction of the 1-cycle $\delta^\pi$.

\begin{theorem}\label{thm:quasiindex_parity_rotational}
  Let $\K$ be a rotational virtual knot and $\pi$ be an index on diagrams of $\K$ with coefficients in an abelian group $A$. Then the formula
\begin{equation}\label{eq:index_delta_rotational}
\hat\delta^\pi_D=\sum_{v\in\V(D)} sgn(v)\pi_D(v)\cdot D^-_v-\frac{w(D)+\tau(D)}{2}\pi^\circ\cdot D
\end{equation}
defines an invariant 1-cycle $\hat\delta^\pi_D$ and the formula
\begin{equation}\label{eq:index_parity_rotational}
\hat p^\pi_D(v)=\sum_{v'\in\V(D)}sgn(v')\pi_D(v')\cdot(D^l_v\cdot D^-_{v'})+\frac{w(D)+\tau(D)}{2}\pi^\circ\cdot ip_D(v)
\end{equation}
defines an oriented parity $\hat p^\pi_D$, $D\in\mathfrak K_{ac}$, on the rotational knot $\K$ with coefficients in the group $\bar A=A/\left<\sigma(\pi)\right>$, where $\sigma(\pi)\in A$ is the signature of the index $\pi$.
\end{theorem}

\section{Intersection formula for long knots}\label{sect:long_knots}

Let us recall the definition of long virtual knots~\cite{MI}.

\begin{definition}\label{def:long_virtual_knot}
A \emph{long virtual knot diagram} is an immersion in general position of the oriented line $\R$ into the plane $\R^2$ which coincides with the axis $Ox$ outside some sufficiently large disc; and each crossing is marked as either classical or virtual crossing, see Fig.~\ref{fig:long_virtual_knot}.

A \emph{long virtual knot} is an equivalence class of long virtual knot diagrams modulo isotopies, Reidemester moves and detour moves.

\begin{figure}[h]
\centering\includegraphics[width=0.5\textwidth]{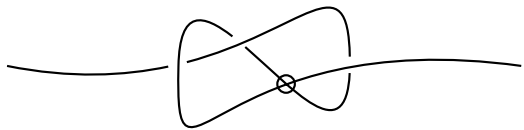}
\caption{A long virtual knot}\label{fig:long_virtual_knot}
\end{figure}
\end{definition}

\begin{definition}\label{def:crossing_order}
Given a crossing $v\in\V(D)$ of a long knot diagram $D$, the oriented smoothing at $v$ splits the diagram $D$ into the \emph{open half} $D^o_v$ of $D$ at the crossing $v$ and the \emph{closed half} $D^c_v$ of $D$ at $v$, see Fig.~\ref{fig:long_knot_halves}.

\begin{figure}[h]
\centering\includegraphics[width=0.6\textwidth]{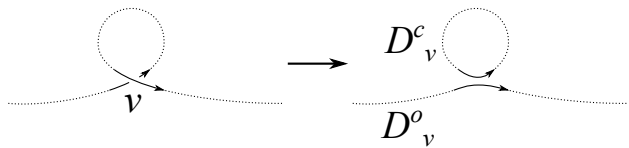}
\caption{Halves of a long knot diagram at the crossing $v$}\label{fig:long_knot_halves}
\end{figure}

For a crossing $v\in\V(D)$ of the diagram $D$ define its \emph{order} $o(v)$ as follows~\cite{CG}: $o(v)=1$ if  $D^c_v=D^r_v$, and $o(v)=-1$ if $D^c_v=D^l_v$.
\end{definition}

For example, the orders of the classical crossings of the diagram in Fig.~\ref{fig:long_virtual_knot} are (from left to right) $-1,1,1$.

\begin{remark}
The results of Sections~\ref{sect:potentials} and~\ref{sect:quasiindex} holds for long knots with two main differences.
\begin{enumerate}
\item Since a long knot is not closed, there is no normalization condition on the parity 1-cycle.
\item Using closed halves, one can localize the increments which the parity quasi-index produces in formula~\eqref{eq:cycle_quasiindex} under Reidemeister moves, and resolve the problem of constructing the parity cycle from the parity quasi-index.
\end{enumerate}
\end{remark}

Below we reformulate the results of Sections~\ref{sect:potentials} and~\ref{sect:quasiindex} for long virtual knots.

\begin{theorem}\label{thm:long_parity_cycle}
1. Let $p$ be an oriented parity with coefficients in an abelian group $A$ on the diagrams of a long virtual knot $\K$. Then formula~\eqref{eq:parity_cycle} defines an invariant 1-cycle $\delta^p_D\in H_1(D,A)$ such that
\begin{equation}\label{eq:intersection_formula_long_knots}
p_D(v)=-o(v)\cdot D^c_v\cdot\delta^p_D.
\end{equation}

2. Given an invariant 1-cycle $\delta$ on diagrams of a long virtual knot $\K$ with coefficients in an abelian group $A$, then the formula
\[
p^\delta_D(v)=-o(v)\cdot D^c_v\cdot\delta_D.
\]
defines an oriented parity $p^\delta$ on on diagrams of $\K$ with coefficients in $A$.
\end{theorem}

\begin{theorem}\label{thm:long_quasiindex_cycle}
1. Let $\delta$ be an invariant 1-cycle on diagrams of a long virtual knot $\K$ with coefficients in an abelian group $A$. Then formula~\eqref{eq:parity_quasiindex} defines a quasi-index $\pi$ on diagrams of $\K$ with coefficients in $A$ such that
\begin{equation}\label{eq:delta_formula_long}
\delta_D=\sum_{v\in\V(D)} o(v)\pi_D(v)\cdot D^c_v+\rho D
\end{equation}
for some $\rho\in A$ and any diagram $D\in\mathfrak K$.

2. Let $\pi$ be a quasi-index on diagrams of a long virtual knot $\K$ with coefficients in an abelian group $A$, and $\rho\in A$ be an arbitrary element. Then formula~\eqref{eq:delta_formula_long} defines an invariant 1-cycle on diagram of the knot $\K$ with coefficients in $A$ and the formula
\begin{equation}\label{eq:long_quasiindex_parity}
p^\pi_D(v)=-o(v)\sum_{v'\in\V(D)}o(v')\pi_D(v')\cdot(D^c_v\cdot D^c_{v'})+o(v)\cdot ip_D(v)\cdot\rho
\end{equation}
defines an oriented parity $p^\pi$ on diagrams of $\K$ with coefficients in $A$.
\end{theorem}

\begin{proof}
1. Let $\delta$ be an invariant 1-cycle on diagram of the knot $\K$ with coefficients in $A$. By Theorem~\ref{thm:quasiindex}, formula~\eqref{eq:cycle_quasiindex} and definition of closed halves, the family of maps $\pi_D$ is a quasi-index and
\[
\delta_D=\sum_{v\in\V(D)} o(v)\pi_D(v)\cdot D^c_v+\hat\rho(D)\cdot D.
\]
Let $a_0$ be the initial arc of the diagram $D$. Since $a_0$ belongs to neither halves $D^c_v$, $v\in\V(D)$, the value of the cycle $\delta_D$ on the arc $a_0$ is equal to $\hat\rho(D)$. Then by invariance of $\delta$, the value $\hat\rho(D)=\delta_D(a_0)$ does not depend on the diagram $D$. Thus, $\hat\rho(D)\equiv\rho\in A$ and we get the formula~\eqref{eq:delta_formula_long}.

2. Let $\pi$ be a quasi-index $\pi$ on diagrams of $\K$ with coefficients in $A$. Define a 1-cycle $\delta$ by formula~\eqref{eq:delta_formula_long}. Let us prove its invariance.

Let $f\colon D\to D'$ be a Reidemeister move. By Proposition~\ref{prop:reminder_reidemeister} and definition of closed halves, we have
\[
\delta_{D'}-\delta_D=\Delta(f)\cdot \hat D\in C_1(\hat D, A)
\]
where $\hat D$ is the intermediate graph in the surface $AD(D)$, see the proof of Proposition~\ref{prop:reminder_reidemeister}. Let $a_0$ ad $a'_0$ be the initial arcs of $D$ and $D'$. Then $f_*(a_0)=a'_0$. Since $a_0$ does not belong to any closed half, $D^c_v(a_0)=0$ for any $v\in\V(D)$ and $\delta_D(a_0)=\rho\cdot D(a_0)=\rho$. Analogously, $\delta_{D'}(a'_0)=\rho$. Hence,
\[
\Delta(f)=\Delta(f)\cdot D(a_0)=\delta_{D'}(a'_0)-\delta_D(a_0)=\rho-\rho=0.
\]
Thus, $\delta_{D'}-\delta_D=0$, and the cycle $\delta$ is invariant.

By Theorem~\ref{thm:long_parity_cycle} the cycle $\delta$ defines an oriented parity, and the formulas~\eqref{eq:intersection_formula_long_knots} and~\eqref{eq:delta_formula_long} imply the formula~\eqref{eq:long_quasiindex_parity}.
\end{proof}

\begin{example}\label{ex:long_knot_quasiindex}
Let $B=\Z$ be the biquandle with operations $x\circ y=x+1$, $x\ast y = x+1$, $x,y\in B$, and $\theta\in H^1(B,\Z)$ be the biquandle 1-cocycle given by the formula $\theta(x)=x$, $x\in B$.

Consider the long knot diagram $D$ in Fig.~\ref{fig:long_virtual_knot}. Choose the colouring $c\in Col_B(D)$ such that $c(a_0)=0$ where $a_0$ is the initial arc of $D$ (Fig.~\ref{fig:long_virtual_knot_coloured}). Then we have a 1-cycle $\delta_D=\delta^\theta_{c,D}$. The corresponding quasi-index is constant $\pi\equiv -1$.

\begin{figure}[h]
\centering\includegraphics[width=0.5\textwidth]{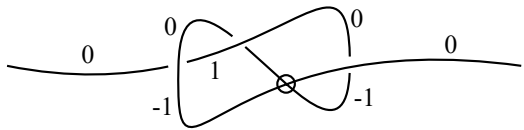}
\caption{A colouring of a long virtual knot}\label{fig:long_virtual_knot_coloured}
\end{figure}

The cycle $\delta_D$ is induced by formula~\eqref{eq:delta_formula_long} with the quasi-index $\pi\equiv -1$ and $\rho=0$.
The values of the parity $p^\pi=p^\theta$ on the crossings are (from left to right) $-1,1,0$ that coincide with the values of the index parity $ip$.
\end{example}

\section{Intersection formula for parity functors}\label{sect:parity_functors}

Let $A$ and $B$ be two abelian groups. A \emph{partial isomorphism} between $A$ and $B$ is a group isomorphism $f\colon A'\to B'$ where $A'\subset A$ and $B'\subset B$ are some subgroups.

Recall the definition of an oriented parity functor~\cite{N1}.

\begin{definition}\label{def:oriented_parity_functor}
An {\em oriented parity functor $P$ on the diagrams of a virtual knot $\mathcal{K}$} is a family of pairs $(P_D, A(D))$, $D\in\mathfrak K$, where $A(D)$ is an abelian group and $P_D\colon\V(D)\to \mathcal A(D)$ is a map from the set of crossings of the diagram $D$, such that for any isotopy or a Reidemeister move $f\colon D\to D'$ a partial isomorphism $A(f)$ between $A(D)$ and $A(D')$ is fixed, and

\begin{itemize}
\item[(P0)] for any corresponding crossings $v\subset \V(D)$ and $v'=f_*(v)\in\V(D')$ we have $P_{D'}(v')=A(f)(P_D(v))$;

\item[(P1)] $P_D(v)=0$ if $f$ is a decreasing first Reidemeister move and $v$ is the disappearing crossing;

\item[(P2)] $P_D(v_1)+P_D(v_2)=0$ if $f$ is a decreasing second Reidemeister move and $v_1,\,v_2$ are the disappearing crossings;

\item[(P3+)] if $f\colon D\to D'$ is a third Reidemeister move then
\[
\epsilon(v_1)P_D(v_1)+\epsilon(v_2)P_D(v_2)+\epsilon(v_3)P_D(v_3)=0
\]
  where $v_1,v_2,v_3$  are the crossings involved in the move and $\epsilon(v_i)$ is the incidence index of the crossing $v_i$ in relation with the disappearing triangle.
\end{itemize}
\end{definition}

Given an oriented parity functor $(P_D,A(D))$, $D\in\mathfrak K$, denote by $\bar A(D)$ the subgroup generated by the elements $P_D(v)$, $v\in\V(D)$, in the coefficient group $A(D)$. The parity functor is called \emph{reduced} if $A(D)=\bar A(D)$ for all diagrams $D$ of the knot $\K$.

\begin{remark}
1. Each parity $p$ with coefficients in a group $A$ is a parity functor with $A(D)=A$, $A(f)=id_A$ and $P_D=p_D$.

2. A parity functor assigns to any diagram its own coefficient group where the parity values lie, and these groups are identified by partial isomorphisms for diagrams connected by a Reidemeister move. This makes parity functors more flexible comparing to parities. For example, parity functors can assimilate monodromy of quasi-indices.
\end{remark}

\begin{theorem}\label{thm:quasiindex_parity_functor}
  Let $\pi$ be a quasi-index with coefficients in an abelian group $A$ on diagrams of a virtual knot $\mathcal K$. Denote $\bar A=A/\left<\sigma(\pi)\right>$ where $\sigma(\pi)$ is the signature of the quasi-index $\pi$. Then the family of maps $P_D\colon \V(D)\to \bar A\oplus\Z$ given by the formula
\begin{equation}\label{eq:quasiindex_parity_functor}
P_D(v)=\left(\sum_{v'\in\V(D)}\pi_D(v')\cdot(D^l_v\cdot D^r_{v'}),\ ip_D(v)\right)
\end{equation}
defines an oriented parity functor on diagrams of the knot $\mathcal K$ with coefficients in the groups $A(D)\equiv \bar A\oplus\Z$. For a Reidemeister move $f\colon D\to D'$ the isomorphism $A(f)$ is defined by the formula
\[
A(f)(x,y)=(x+\Delta_\pi(f)\cdot y,y)
\]
where $\Delta_\pi(f)=\rho(D')-\rho(D)$ is determined by Proposition~\ref{prop:reminder_reidemeister}.
\end{theorem}
\begin{proof}
The properties (P1)--(P3+) hold because the right-hand side of the formula~\eqref{eq:quasiindex_parity_functor} is the intersection of the half $D^l_v$ with a normalized 1-cycle.

Let us check the property (P0). Let $f\colon D\to D'$ be a Reidemeister move, $v\in\V(D)$ and $v'=f_*(v)\in\V(D')$ be corresponding crossings. Denote $P_D=(p_D,ip_d)$. By Proposition~\ref{prop:reminder_reidemeister} we have
\[
p_{D'}(v')=p_D(v)-\Delta_\pi(f)\cdot(D^l_v\cdot D)=p_D(v)+\Delta_\pi(f)\cdot ip_D(v).
\]
Then
\[
A(f)(p_D(v),ip_D(v))=\left(p_D(v)+\Delta_\pi(f)\cdot ip_D(v),ip_D(v)\right)=(p_{D'}(v'),ip_{D'}(v')),
\]
i.e. $P_{D'}(v')=A(f)(P_D(v))$.
\end{proof}

Let us show how intersection formula manifests itself for parity functors.

Let $(P_D,A(D))$ be an oriented parity functor on diagrams of a virtual knot $\K$. Given a diagram $D$ of the knot $\K$, construct the \emph{extended diagram} $D_{ext}$ by applying second Reidemeister moves as shown in Fig.~\ref{fig:potential_delta} on every arc of the diagram $D$. The coefficient group $A(D)_{ext}=A(D_{ext})$ is called the \emph{extended coefficient group} of $D$. Then the formula~\eqref{eq:parity_cycle} defines a 1-cycle
\[
\delta^P_D=\sum_{a\in\mathcal A(D)}\delta^P_a\cdot a\in C_1(D,A(D)_{ext}).
\]
Then the formula~\eqref{eq:cycle_quasiindex} defines a quasi-index $\pi^P_D\colon\V(D)\to A(D)_{ext}$.

\begin{remark}
1. Let $D$ be a diagram of the knot $\K$. By definition, there is a sequence of increasing second Reidemeister moves $e_D\colon D\to D_{ext}$. The morphism $e_D$ induces the monomorphism $A(e_D)\colon \bar A(D)\hookrightarrow\bar A(D)_{ext}$. The subgroup $\bar A(D)_{ext}$ is generated by the elements $A(e_D)(P_D(v))$, $v\in\V(D)$, which correspond to the crossings of the diagram $D$, and  the parities $P_{D_{ext}}(v')$, $v'\in\V(D_{ext})\setminus (e_D)_*(\V(D))$, which can be identified with the values $\delta^P_a$, $a\in\mathcal A(D)$, and correspond to the arcs of the diagram $D$.

2. If $f\colon D\to D'$ is an increasing first or second Reidemeister move, then there is a sequence of increasing first and second Reidemeister moves $f_{ext}\colon D_{ext}\to D'_{ext}$, see Fig~\ref{fig:r1_r2_ext}. Then the map $f_{ext}$ induces a monomorphism $A(f_{ext})\colon \bar A(D)_{ext}=\bar A(D_{ext})\to \bar A(D'_{ext})=\bar A(D')_{ext}$.

\begin{figure}[h]
\centering\includegraphics[width=0.5\textwidth]{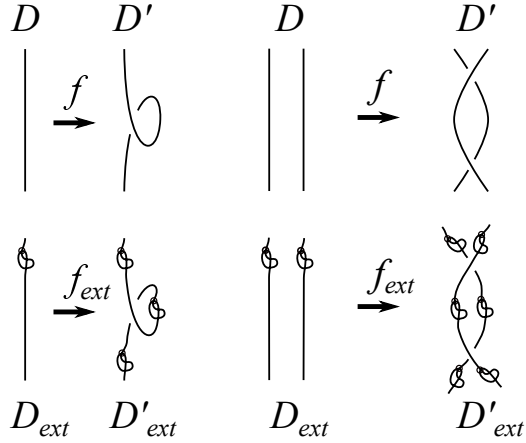}
\caption{Extended first and second Reidemeister moves}\label{fig:r1_r2_ext}
\end{figure}

The subgroup $\bar A(D)_{ext}$ includes the values $\delta^P_a$, $a\in\mathcal A(D)$, so we can use the map $A(f_{ext})$ to formulate the invariance condition for the cycle $\delta^P_D$ (see Theorem~\ref{thm:parity_cycle_quasiindex_functor}).

3. Let $f\colon D\to D'$ be a third Reidemeister move. Then there is no a direct Reidemeister move between the diagrams $D_{ext}$ and $D'_{ext}$ but there are intermediate diagrams $\bar D_{ext}$ and $\bar D'_{ext}$ (Fig.~\ref{fig:r3_ext}) connected by a third Reidemeister move $\bar f_{ext}\colon \bar D_{ext}\to \bar D'_{ext}$. The diagrams $D_{ext}$ and $D'_{ext}$ are obtained from $\bar D_{ext}$ and $\bar D'_{ext}$ by applying increasing second Reidemeister moves on the arcs $x,y,z$ and $x',y',z'$ of the triangle of the move $f$.

\begin{figure}[h]
\centering\includegraphics[width=0.9\textwidth]{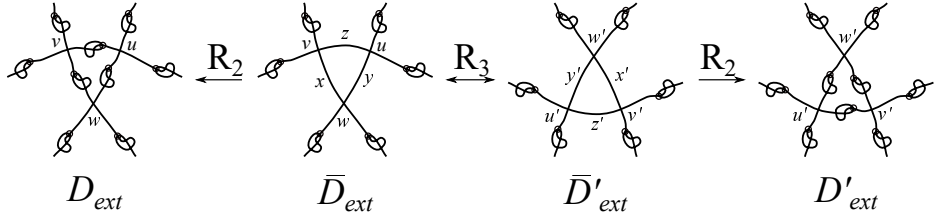}
\caption{Extended third Reidemeister move}\label{fig:r3_ext}
\end{figure}

The move $\bar f_{ext}$ induces the group isomorphism
\[
A(\bar f_{ext})\colon \bar A(D)_{ext}=\bar A(\bar D_{ext})\to\bar A(\bar D'_{ext})=\bar A(D')_{ext}.
\]
The moves $\bar D_{ext}\to D_{ext}$ allows to identify the group $\bar A(D)_{ext}$ with the subgroup in $A(D_{ext})$ generated by values $\delta^P_a$ for all arcs $a\in \mathcal A(D)$ except the arcs $x,y,z$. An analogous statement holds for $\bar A(D')_{ext}$. Thus, the map $A(\bar f_{ext})$ allows to identify the values $\delta^P_a$, $a\in\mathcal A(D)\setminus\{x,y,z\}$, with the values $\delta^P_{a'}$, $a'\in\mathcal A(D')\setminus\{x',y',z'\}$.

Hence, $A(\bar f_{ext})$ induces a correspondence between the quasi-index values $\pi_D(t)$ and $\pi_{D'}(t')$ for all crossings which do not take part in the move $f$. Note that, $\pi_D(u)\in A(D_{ext})$ does not belong to the subgroup $\bar A(D)_{ext}$ in general because it is expressed using the values $\delta^P_y$ and $\delta^P_z$. Then the value $\pi_D(u)$ can not be compared with the value $\pi_{D'}(u')$. By the same reason, one can not compare $\pi_D(v)$ with $\pi_{D'}(v')$ and $\pi_D(w)$ with $\pi_{D'}(w')$.

On the other hand, $\epsilon(u)\pi_D(u)-\epsilon(v)\pi_D(v)\in \bar A(D)_{ext}$, so we can compare its image under the isomorphism $A(\bar f_{ext})$ with the element  $\epsilon(u)\pi_{D'}(u')-\epsilon(v)\pi_{D'}(v')$ (see Theorem~\ref{thm:parity_cycle_quasiindex_functor}).

This fact gives an explanation why for the maps $\pi_D\colon\V(D)\to \bar A(D)_{ext}$ it is natural to consider the conditions of quasi-indices but not of indices.
\end{remark}

For parity functors, we have the following analogue of Theorems~\ref{thm:parity_cycle} and~\ref{thm:quasiindex}.

\begin{theorem}\label{thm:parity_cycle_quasiindex_functor}
Let $(P_D,A(D))$ be an oriented parity functor on diagrams of a virtual knot $\K$. Then the chain $\delta^P_D\in C_1(D,A(D)_{ext})$ defined above possesses the properties
\begin{itemize}
  \item $\delta^P_D$ is a cycle, i.e. $\delta^P_D\in H_1(D,A(D)_{ext})$;
  \item $\delta^P_D$ is invariant, i.e. for any Reidemeister move $f\colon D\to D'$ and any correspondent arcs $a\in\mathcal A(D)$ and $a'\in\mathcal A(D')$ one has
\[
A(f_{ext})(\delta^P_D(a))=\delta^P_{D'}(a');
\]
  \item $\delta^P_D$ is normalized, i.e. $D\cdot\delta^P_D=0$.
\end{itemize}

The family of maps $\pi^P_D\colon\V(D)\to A(D)_{ext}$ defined above possesses the properties
\begin{itemize}
  \item[(Q0)] for any Reidemeister move $f\colon D\to D'$ and any correspondent crossings $v\in\V(D)$ and $v'\in\V(D')$ one has $A(f_{ext})(\pi^P_D(v))=\pi^P_{D'}(v')$;
\item[(Q2)] $\pi_D(v_1)=\pi_D(v_2)$ for any crossings $v_1,\,v_2\in\V(D)$ to which a decreasing second Reidemeister move can be applied;

\item[(Q3)] if $v_1,v_2,v_3\in\V(D)$ are the crossings which take part in a third Reidemeister move $f\colon D\to D'$ and $v'_1,v'_2,v'_3$ are the corresponding crossings in $D'$ then for any $i,j=1,2,3$ one has
\[
A(f_{ext})\left(\epsilon(v_i)\pi^P_D(v_i)-\epsilon(v_j)\pi^P_D(v_j)\right)=\epsilon(v'_i)\pi^P_{D'}(v'_i)-\epsilon(v'_j)\pi^P_{D'}(v'_j),
\]
where  $\epsilon_\Delta(v_i)$ is the incidence index of the crossing $v_i$ to the disappearing triangle $\Delta$.
\end{itemize}

The quasi-index $\pi^P$ and the cycle $\delta^P$ are connected by the formula
\[
\delta^P_D=\sum_{v\in\V(D)} \pi^P_D(v)\cdot D^r_v+\rho(D)\cdot D\in H_1(D,A(D)_{ext})
\]
 with a unique  $\rho(D)\in A(D)_{ext}$.
\end{theorem}

Let us formulate an inverse statement.

\begin{definition}\label{def:crossing_tribe}
Let $\K$ be a virtual knot and $\mathfrak K$ be the category of its diagrams. A \emph{tribal system} is a family of partitions $\mathfrak C(D)$ of the sets $\V(D)$, $D\in\mathfrak K$, into subsets $C\in\mathfrak C(D)$ called \emph{tribes} such that
\begin{itemize}
  \item[(T0)] Let $f\colon D\to D'$ be a Reidemeister move, $v_1,v_2\in\V(D)$ and $v'_1=f_*(v_1)$ and $v'_2=f_*(v_2)$ be the correspondent crossings in $D'$. Then if $v_1$ and $v_2$ belong to one tribe in $\mathfrak C(D)$ then $v'_1$ and $v'_2$ also belong to one tribe in $\mathfrak C(D')$;
  \item[(T2)] If $f\colon D\to D'$ is a decreasing second Reidemeister move and $v_1$ and $v_2$ are the disappearing crossings then $v_1$ and $v_2$ belong to one tribe in $\mathfrak C(D)$.
\end{itemize}
\end{definition}

\begin{remark}
1. Tribal system are in the same sort of relationship with indices as parity functors with parities. Given a quasi-index $\pi$ on diagrams of a virtual knot $\K$ with coefficients in a group $A$, consider the partitions $\mathfrak C^\pi(D)=\{\pi_D^{-1}(x)\}_{x\in A}$ where $\pi_D\colon\V(D)\to A$, $D\in\mathfrak K$, is the quasi-index map. Then $\mathfrak C^\pi$ is a tribal system on the diagrams of the knot $\K$.

2. Property (T0) implies that any Reidemeister move $f\colon D\to D'$ induces a partial bijection $f_{\mathfrak C}\colon\mathfrak C(D)\to\mathfrak C(D')$ where $f_{\mathfrak C}(C)=C'$ if there exists $v\in C\in\mathfrak C(D)$ such that $f_*(v)\in C'\in\mathfrak C(D')$.

3. Given a tribal system $\mathfrak C$, the sign function splits any tribe $C\in\mathfrak C(D)$, $D\in\mathfrak K$, into two subsets $C^+$ and $C^-$ we call \emph{phratries}. By definition
\[
C^\pm=\{v\in C\mid sgn(v)=\pm 1\}.
\]
The phratries $C^+$ and $C^-$ are called \emph{dual}.

We can refine the properties (T0) and (T2) as follows:
\begin{itemize}
  \item[($\Phi0$)] for any Reidemeister move $f\colon D\to D'$ and any crossings $v_1, v_2\in\V(D)$ such that the crossings $f_*(v_1)$ and $f_*(v_2)$ exist, if $v_1$ and $v_2$ belong to one phratry then $f_*(v_1)$ and $f_*(v_2)$ belong to one phratry too.
  \item[($\Phi2$)] If $f\colon D\to D'$ is a decreasing second Reidemeister move and $v_1$ and $v_2$ are the disappearing crossings then $v_1$ and $v_2$ belong to dual phratries.
\end{itemize}
\end{remark}

Let $\mathfrak C$ be a tribal system on the diagrams of a virtual knot $\K$. Consider the family of maps
$\pi^{\mathfrak C}_D\colon \V(D)\to\mathfrak C(D)$, $D\in\mathfrak K$, given by the formula $\pi^{\mathfrak C}_D(v)=C$ if $v\in C\in\mathfrak C(D)$. The maps $\pi^{\mathfrak C}_D$ satisfy the conditions
\begin{itemize}
  \item for any Reidemeister move $f\colon D\to D'$ and any correspondent crossings $v\in\V(D)$ and $v'=f_*(v)\in\V(D')$ we have $\pi^{\mathfrak C}_{D'}(v')=f_{\mathfrak C}(\pi^{\mathfrak C}_D(v))$;
  \item If $f\colon D\to D'$ is a decreasing second Reidemeister move and $v_1$ and $v_2$ are the disappearing crossings then $\pi^{\mathfrak C}_D(v_1)=\pi^{\mathfrak C}_D(v_2)$.
\end{itemize}

Denote the groups $\Z[\mathfrak C(D)]$ by $A(D)$. For any Reidemeister move $f\colon D\to D'$ the partial bijection $f_*\colon\mathfrak C(D)\to\mathfrak C(D')$ induces a partial isomorphism $A(f)\colon A(D)\to A(D')$.

Define the \emph{signature} of the tribal system $\mathfrak C$ as the elements
\begin{equation}\label{eq:tribal_system_signature}
  \sigma(\mathfrak C)_D=-\sum_{v\in\V(D)}\pi^{\mathfrak C}_D(v)\cdot ip_D(v)\in A(D), D\in\mathfrak K.
\end{equation}
Then for any Reidemeister move $f\colon D\to D'$ we have $A(f)(\sigma(\mathfrak C)_D)=\sigma(\mathfrak C)_{D'}$.

Consider the groups $\hat A(D)=A(D)/\left<\sigma(\mathfrak C)_D\right>\oplus\Z$, $D\in\mathfrak K$.

Given a Reidemeister move $f\colon D\to D'$, define partial isomorphisms $\hat A(f)\colon \hat A(D)\to \hat A(D')$ by the formulas:
\begin{itemize}
\item if $f$ is an increasing first Reidemeister move and $v'\in\V(D')$ is the new crossing then we set
\[
\hat A(f)(x,y)=\left(A(f)(x)+k(f)\cdot\pi^{\mathfrak C}_{D'}(v')y,y\right), \quad (x,y)\in\hat A(D),
\]
where $k(f)=0$ if the loop crossing $v'$ is of types $r_+$ or $r_-$, and $k(f)=1$ if $v'$ is of type $l_+$ or $l_-$;

\item if $f$ is a decreasing first Reidemeister move and $v\in\V(D)$ is the disappearing crossing then we set
\[
\hat A(f)(x,y)=\left(A(f)(x-k(f)\cdot\pi^{\mathfrak C}_{D}(v)y),y\right), \quad (x,y)\in\hat A(D),
\]
where $k(f)=0$ if the loop crossing $v$ is of types $r_+$ or $r_-$, and $k(f)=1$ if $v$ is of type $l_+$ or $l_-$;

\item if $f$ is an increasing second Reidemeister move and $v'_1,v'_2\in\V(D')$ are the new crossings then we set
\[
\hat A(f)(x,y)=\left(A(f)(x)+\pi^{\mathfrak C}_{D'}(v'_1)y,y\right), \quad (x,y)\in\hat A(D);
\]

\item if $f$ is a decreasing second Reidemeister move and $v_1,v_2\in\V(D)$ are the disappearing crossings then we set
\[
\hat A(f)(x,y)=\left(A(f)(x-\pi^{\mathfrak C}_{D}(v_1)y),y\right), \quad (x,y)\in\hat A(D);
\]

\item if $f$ is a third Reidemeister move then we set
\[
\hat A(f)(x,y)=\left(A(f)(x),y\right), \quad (x,y)\in\hat A(D).
\]
\end{itemize}

Note that the map $\hat A(f)$ is defined not for all elements $(x,y)\in\hat A(D)$.

\begin{theorem}\label{thm:tribal_system_cycle_parity}
  Let $\mathcal C$ be a tribal system on the diagrams of a virtual knot $\K$. Then the formula
\begin{equation}\label{eq:tribal_system_cycle}
  \hat\delta^{\mathfrak C}_D=\sum_{v\in\V(D)}\pi^{\mathfrak C}_D(v)\cdot D^r_v \oplus D\in C_1(D,\hat A(D))
\end{equation}
defines a normalized invariant 1-cycle (in the sense of Theorem~\ref{thm:parity_cycle_quasiindex_functor}).

The formula
\begin{equation}\label{eq:tribal_system_parity_functor}
P^{\mathfrak C}_D(v)=\left(\sum_{v'\in\V(D)}\pi^{\mathfrak C}_D(v')\cdot(D^l_v\cdot D^r_{v'}),\ -ip_D(v)\right)
\end{equation}
defines an oriented parity functor on diagrams of the knot $\mathcal K$ with coefficients in the groups $\hat A(D)$.
\end{theorem}

The theorem can proved with the reasoning of Theorem~\ref{thm:quasiindex_parity_functor}.

\section{Parity cycle and parities on virtual links}\label{sect:links}

Let $\mathcal L=K_1\cup\dots\cup K_d$ be an oriented virtual link with $d$ components, and $\mathfrak L$ be the category of its diagrams. Let $A$ be an abelian group.

Let $p$ be an oriented parity on the diagrams of the link $\mathcal L$ with coefficients in the group $A$.

Given a diagram $D\in\mathfrak L$, one can define the parity cycle $\delta^p_D$ by the formula~\eqref{eq:parity_cycle}. Using the reasonings of Theorem~\ref{thm:parity_cycle}, one can prove that $\delta^p$ is a normalized invariant cycle. The normalization condition for links takes the form:
\[
D_i\cdot\delta^p_D=0
\]
for any component $D_i$ of the link diagram $D$.

On the other hand, the intersection formula~\eqref{eq:intersection_formula} does not work in general because there is no notion of a knot half for a mixed crossing in a link. Nevertheless, the parity cycle defines some kind of relative parity (for relative parities see~\cite{KryM}). More precisely, the following theorem hold.

\begin{theorem}\label{thm:relative_parity_cycle}
Let $D=D_1\cup\dots\cup D_d$ be a diagram of an oriented virtual link $\mathcal L$, and $p$ be an oriented parity on the diagrams of the link $\mathcal L$ with coefficients in an abelian group $A$. Let $\delta^p$ be the parity cycle of $p$.
\begin{enumerate}
   \item For any self crossing $v$ of a component $D_i$ its parity is equal
\begin{equation}\label{eq:intersection_formula_selfcrossing}
  p_D(v)=(D_i)^l_v\cdot\delta^p_D
\end{equation}
where $(D_i)^l_v$ is the left half of the component $D_i$ at the crossing $v$.
   \item For any crossings $v$ and $w$ of components $D_i$ and $D_j$ denote the path from $v$ to $w$ in $D_i$ (along the orientation of the component) by $\gamma_1$, and the path from $w$ to $v$ in $D_j$ by $\gamma_2$ (Fig.~\ref{fig:relative_parity}). Let $\gamma=\gamma_1\gamma_2$. Then
\begin{equation}\label{eq:intersection_formula_relative_parity}
  \eta_\gamma(v) p_D(v)+\eta_\gamma(w)p_D(w)=\gamma\cdot\delta^p_D
\end{equation}

\begin{figure}[h]
\centering\includegraphics[width=0.4\textwidth]{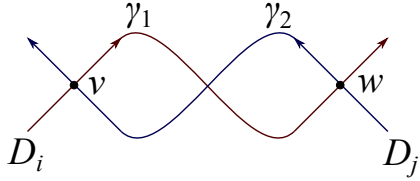}
\caption{Paths $\gamma_1$ and $\gamma_2$}\label{fig:relative_parity}
\end{figure}

where $\eta_\gamma(v)$ and $\eta_\gamma(w)$ are the incidence indices of the crossings $v$ and $w$ to the cycle $\gamma$ (see Fig.~\ref{fig:eta_incidence}).
 \end{enumerate}

 \begin{figure}[h]
\centering\includegraphics[width=0.15\textwidth]{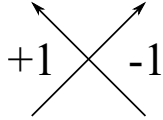}
\caption{Incidence index $\eta$}\label{fig:eta_incidence}
\end{figure}
\end{theorem}

\begin{proof}
The proof of the first statement of the theorem repeats the proof in Theorem~\ref{thm:parity_cycle}.

Like in Theorem~\ref{thm:parity_cycle}, fix an arc $o$ of the diagram $D$ and define the potential function $\phi(a)=\delta_{a,o}$ on the arcs of $D$.

Let $v$ and $w$ of components $D_i$ and $D_j$. Denote arcs incident to $v$ and $w$ by $a_1$, $d_1$, $a_2$, $d_2$ as shown in Fig.~\ref{fig:relative_parity_proof}. By Proposition~\ref{prop:parity_potential} $p_D(v)=\phi(a_1)-\phi(d_1)$ and $p_D(w)=\phi(a_2)-\phi(d_2)$. By the proof of Theorem~\ref{thm:parity_cycle} $\gamma_1\cdot\delta^p_D=\phi(a_2)-\phi(a_1)$ and $\gamma_2\cdot\delta^p_D=\phi(d_1)-\phi(d_2)$. Then
\begin{multline*}
\gamma\cdot\delta^p_D =(\gamma_1+\gamma_2)\cdot\delta^p_D=\phi(a_2)-\phi(d_2)+\phi(d_1)-\phi(a_1)=\\
p_D(w)-p_D(v)= \eta_\gamma(v) p_D(v)+\eta_\gamma(w)p_D(w)
\end{multline*}
because $\eta_\gamma(v)=-1$ and $\eta_\gamma(w)=1$.

\begin{figure}[h]
\centering\includegraphics[width=0.4\textwidth]{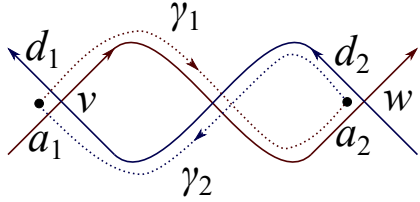}
\caption{Proof of the intersection formula~\eqref{eq:intersection_formula_relative_parity}}\label{fig:relative_parity_proof}
\end{figure}
The proof for other configurations is analogous.
\end{proof}

\begin{example}[Link parity]\label{exa:link_parity_general}
Let $\mathcal L=K_1\cup\dots\cup K_d$ be an oriented virtual link with $d$ components, and $A$ be an abelian group. Choose an arbitrary $(d-1)$-vector $l=(l_1,\dots, l_{d-1})$ in $A^{d-1}$. Denote also $l_d=0$.

Let $D=D_1\cup\dots\cup D_d$ be a diagram of the link $\mathcal L$. For a crossing $v\in\V(D)$ of components $D_i$ and $D_j$ (Fig.~\ref{fig:link_crossing}), its \emph{link parity} $lp^l_D(v)$ is defined by the formula
\begin{equation}\label{eq:link_parity_g}
  lp^l_D(v)=l_i-l_j.
\end{equation}

\begin{figure}[h]
\centering\includegraphics[width=0.15\textwidth]{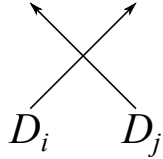}
\caption{Link crossing}\label{fig:link_crossing}
\end{figure}

Then by definition $lp^l(v)=0$ for any self-crossing if $D$. Hence, the parity cycle of the link parity is trivial: $\delta^{lp^l}=0$. Thus, a parity on links can not be restored from the parity cycle in general.
\end{example}

We can specify relation between parities and parity cycles with the following proposition.

\begin{proposition}\label{prop:link_parity_kernel}
Let $\mathcal L$ be an oriented virtual link with $d$ components, and let $A$ be an abelian group. Let $p_1, p_2$ be  oriented parities on the diagrams of the link $\mathcal L$ with coefficients in the group $A$, and $\delta^{p_1}$ and $\delta^{p_2}$ be the parity cycles. Then $\delta^{p_1}=\delta^{p_2}$ if and only if $p_2=p_1+lp^l$ for some $l\in A^{d-1}$.
\end{proposition}
\begin{proof}
If $p_2=p_1+lp^l$ then $\delta^{p_2}=\delta^{p_1}+\delta^{lp^l}=\delta^{p_1}+0=\delta^{p_1}$.

Assume $\delta^{p_1}=\delta^{p_2}$ then the parity $p=p_2-p_3$ has zero parity cycle. For any $i,j\in\{1,\dots, n\}$ take a diagram $D$ of the link $\mathcal L$ such that the components $D_i$ and $D_j$ intersect at a crossing $v$. Denote $l_{ij}=p_D(v)$. By property (P0) and Theorem~\ref{thm:relative_parity_cycle} the value $l_{ij}$ does not depend on the choice of the crossing $v$ and the diagram $D$. Thus, the parity of any crossing depends only on the components which intersect at this crossing, and is equal to some $l_{ij}$.

Property (P3+) applied to the configuration in Fig.~\ref{fig:link_3crossing} implies that $l_{ij}+l_{jk}+l_{ki}=0$ for all $i,j,k$.

\begin{figure}[h]
\centering\includegraphics[width=0.2\textwidth]{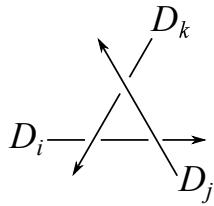}
\caption{Intersections of components $D_i$, $D_j$, $D_k$}\label{fig:link_3crossing}
\end{figure}

Denote $l_i=l_{id}$, $i=1,\dots,d$. Then $l_d=0$ and $l_{ij}=l_i-l_j$ for all $i,j=1,\dots,d$. Thus, $p=lp^l$ with $l=(l_1,\dots,l_{d-1})$.
\end{proof}

\begin{remark}
We can reformulate the proposition above as follows (cf. Remark~\ref{rem:parity_cycle_isomorphism}). Let $\mathcal{P}(\mathcal L,A)$ be the set of oriented parities with coefficients in $A$ on the diagrams of the link $\mathcal L$ and $\mathcal{NIC}(\mathcal L,A)$ be the set of normalized invariant cycles with coefficients in $A$ on the diagrams of $\mathcal L$. The formula~\eqref{eq:parity_cycle} defines a homomorphism
\[
\Delta\colon \mathcal{P}(\mathcal L,A)\to\mathcal{NIC}(\mathcal L,A).
\]
Proposition~\ref{prop:link_parity_kernel} states that $\ker\Delta=\mathcal{LP}(\mathcal L,A)$ is the subgroup consisting of link parities. Note that $\mathcal{LP}(\mathcal L,A)\simeq A^{d-1}$.

On the other hand, the example below shows that in general the map $\Delta$ is not an epimorphism.
\end{remark}

\begin{example}\label{exa:link_cycle_counterexample}
Let $D$ be the diagram of a link $\mathcal L$ with three components shown in Fig.~\ref{fig:relative_parity_example} left. Consider a constant 1-cycle $\delta$ on $\mathcal L$ with coefficients in $\Z_2$: $\delta_D(a)=1$ for any arc $a\in\mathcal A(D)$. Then $\delta$ is a normalized invariant cycle.
\begin{figure}[h]
\centering\includegraphics[width=0.9\textwidth]{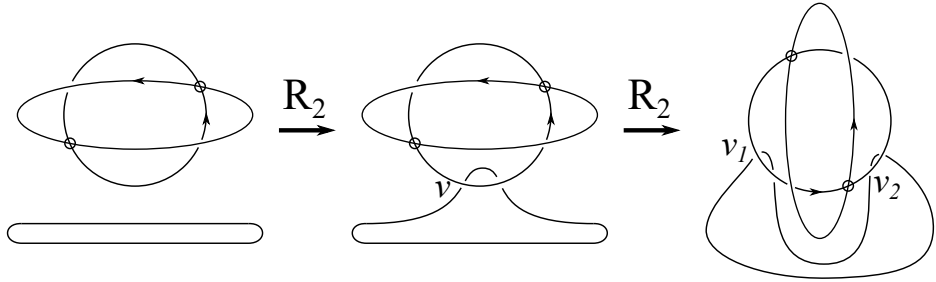}
\caption{Link diagrams}\label{fig:relative_parity_example}
\end{figure}

Assume that there is a parity $p$ with coefficients in $\Z_2$ such that $\delta=\delta^p$. Apply a second Reidemeister move to get a diagram $D'$ with a crossing $v$ (Fig.~\ref{fig:relative_parity_example} middle). There are two second Reidemeister moves $f_1$ and $f_2$ from $D'$ to the diagram $D''$ shown in Fig.~\ref{fig:relative_parity_example} right such that $(f_i)_*(v)=v_i$, $i=1,2$. Then by property (P0) $p_{D''}(v_1)=p_{D'}(v)=p_{D''}(v_2)$. But by the formula~\eqref{eq:intersection_formula_relative_parity} we have $p_{D''}(v_1)-p_{D''}(v_2)=1$. Thus, there is no parity whose parity cycle is $\delta$.

Note that this contradiction can be avoided if one considers parity functors.
\end{example}

\section{Intersection formula for flat knots}\label{sect:flat_knots}

Recall that a \emph{flat link} is an equivalence class of virtual diagrams modulo isotopies, Reidemeister and detour moves, and crossing switches (Fig.~\ref{fig:crossing_switch}). In other words, flat links ignore the information on under- and overcrossing.
\begin{figure}[h]
\centering\includegraphics[width=0.3\textwidth]{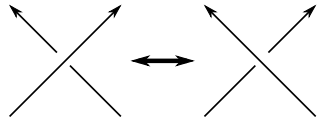}
\caption{Crossing switch}\label{fig:crossing_switch}
\end{figure}

The reader may be noted that most of results of Sections~\ref{sect:potentials} and~\ref{sect:quasiindex} (including Theorems~\ref{thm:parity_cycle}, \ref{thm:parity_cycle_inverse}, \ref{thm:quasiindex} and~\ref{thm:quasiindex_parity}) hold for flat knots. On the other hand the formula~\eqref{eq:quasiindex_cycle_base} and the derived ones do not fit for flat knots because flat knots have not the sign function.

Let $\K$ be a flat knot, and $\pi$ be an index on the diagrams of the knot $\K$ with coefficients in an abelian group $A$. Note that loop crossings (Fig.~\ref{fig:loop_types}) in flat knot diagrams has only one index value that we denote by $\pi^\circ\in A$. Let us formulate an analogue of Corollary~\ref{cor:index_parity_reduced} for flat knots.

\begin{theorem}
Let $\pi$ be an index on diagrams of a flat knot $\K$ with coefficients in an abelian group $A$ and let $\pi^\circ\in A$ be the index value of the loop crossings. Then the formula
\begin{equation}\label{eq:index_flat_reduction_delta}
\bar\delta^\pi_D=\sum_{v\in\V(D)\colon \pi_D(v)\ne\pi^\circ}\pi_D(v)\cdot D^r_v-\sum_{x\in A\setminus\{\pi^\circ\}}\left\lfloor\frac{|\pi^{-1}_D(x)|}{2}\right\rfloor\cdot D
\end{equation}
 defines an invariant 1-cycle $\bar\delta^\pi_D$ and the formula
\begin{equation}\label{eq:index_reduction_parity_flat}
\bar p^\pi_D(v)=\sum_{v'\in\V(D)\colon \pi_D(v')\ne\pi^\circ}\pi_D(v')\cdot(D^l_v\cdot D^r_{v'})+\sum_{x\in A\setminus\{\pi^\circ\}}\left\lfloor\frac{|\pi^{-1}_D(x)|}{2}\right\rfloor\cdot ip_D(v)
\end{equation}
defines an oriented parity $\bar p^\pi$ on diagrams of the flat knot $\K$ with coefficients in the group $\bar A=A/\left<\bar\sigma(\pi)\right>$, where
\[
\bar\sigma(\pi)=D\cdot\bar\delta^\pi_D=-\sum_{v\in\V(D)\colon \pi_D(v)\ne\pi^\circ}\pi_D(v)\cdot ip_D(v)
\]
and $\lfloor\cdot\rfloor$ if the floor function.
\end{theorem}

\begin{proof}
Property (I0) of the index $\pi$ implies invariance of $\bar\delta^\pi$ by isotopies and third Reidemeister moves. By definition, $\bar\delta^\pi$ is invariant under first Reidemeister moves.

Let $f\colon D\to D'$ be an increasing second Reidemeister move and $v'_1, v'_2$ be the new crossings. If $\pi_{D'}(v'_1)=\pi_{D'}(v'_2)=\pi^\circ$ then the expressions for $\bar\delta^\pi_D$ and $\bar\delta^\pi_{D'}$ coincide up to replacement $D$ with $D'$ and $D^r_v$ with $D'R_{f_*(v)}$. Thus, $\bar\delta^\pi_D(a)=\bar\delta^\pi_{D'}(a')$ for any correspondent arcs $a\in\mathcal A(D)$ and $a'\in\mathcal A(D')$.

Let $\pi_{D'}(v'_1)=\pi_{D'}(v'_2)\ne\pi^\circ$. For all $x\in A$, $x\ne\pi_{D'}(v_1)$,  we have $|\pi^{-1}_{D'}(x)|=|\pi^{-1}_D(x)|$, and
$|\pi^{-1}_{D'}(\pi_{D'}(v'_1))|=|\pi^{-1}_D(\pi_{D'}(v'_1))|+2$. Then the difference between $\bar\delta^\pi_{D'}$ and $\bar\delta^\pi_{D}$ is equal to
\[
\pi_{D'}(v'_1)(D'^r_{v_1}+D'^r_{v_1}-D).
\]
Since the chain $D'^r_{v_1}+D'^r_{v_1}-D$ vanishes outside the region where the move $f$ occurs, we have $\bar\delta^\pi_D(a)=\bar\delta^\pi_{D'}(a')$ for any correspondent arcs $a\in\mathcal A(D)$ and $a'\in\mathcal A(D')$.
\end{proof}

\begin{remark}
If $A$ is a ring and $\frac 12\in A$ then we can omit the rounding in~\eqref{eq:index_flat_reduction_delta},~\eqref{eq:index_reduction_parity_flat} and write the following formulas for the invariant 1-cycle and the parity induced by an index $\pi$:
\begin{equation}\label{eq:index_flat_reduction_delta12}
\bar\delta^\pi_D=\frac 12\sum_{v\in\V(D)\colon \pi_D(v)\ne\pi^\circ}\pi_D(v)\cdot (D^r_v-D^l_v)
\end{equation}
and
\begin{equation}\label{eq:index_reduction_parity_flat12}
\bar p^\pi_D(v)=\frac 12\sum_{v'\in\V(D)\colon \pi_D(v')\ne\pi^\circ}\pi_D(v')(D^l_v\cdot (D^r_v-D^l_v)).
\end{equation}

\end{remark}

\section{Derived parities of virtual knots}\label{sect:derived_parity}

In this section we look at connection between parities and indices from another (and more direct) side. This will allow us to loop the intersection formula and define derived parities.

\subsection{Signed index}\label{subsect:signed_index}
Note that an (oriented) parity $p$ is not an index in general. Let us introduce an intermediate notion between parities and indices. Let $\K$ be a virtual knot and $\mathfrak K$ be the category of its diagrams. Let $A$ be an abelian group.

\begin{definition}\label{def:signed_index}
 A \emph{signed index $\tau$ on the diagrams of the knot $\K$ with coefficients in $A$} is a family of maps $\tau_D\colon\V(D)\to A$, $D\in\mathfrak K$, which obeys the properties
\begin{itemize}
\item[(I0)] for any Reidemeister move $f\colon D\to D'$ and any crossings $v\in\V(D)$ and $v'\in\V(D')$ such that $v'=f_*(v)$, one has $\tau_{D}(v)=\tau_{D'}(v')$;
\item[(I2+)] $\tau_D(v_1)=-\tau_D(v_2)$ for any crossings $v_1,\,v_2\in\V(D)$ to which a decreasing second Reidemeister move can be applied.
\end{itemize}
\end{definition}

Let us list some simple properties of signed indices.

\begin{proposition}\label{prop:signed_index_properties}
\begin{enumerate}
  \item Any oriented parity $p$ is a signed index;
  \item The sign function $sgn$ is a signed index with coefficients in $\Z$;
  \item If $\tau_1$ is a signed index with coefficients in a group $A$ and $\tau_1$ is a signed index with coefficients in a group $B$ then their product $\tau_1\otimes\tau_2$ where
\[
(\tau_1\otimes\tau_2)_D(v)=(\tau_1)_D(v)\otimes(\tau_2)_D(v)\in A\otimes B,\ v\in\V(D),
\]
is an index with coefficients in $A\otimes B$;
  \item If $\tau$ is a signed index with coefficients in a group $A$ and $\pi$ is an index with coefficients in a group $B$ then their product $\tau\otimes\pi$ is a signed index with coefficients in $A\otimes B$.
\end{enumerate}
Note that if $\pi_1$ is an index with coefficients in $A$ and $\pi_1$ is an index with coefficients in $B$ then $\pi_1\otimes\pi_2$ is an index with coefficients in $A\otimes B$.
\end{proposition}

\begin{definition}
Let $A$ and $B$ be groups. A map $f\colon A\to B$ is called \emph{odd} if for any $x\in A$  $f(-x)=-f(x)$ in $B$, and the map $f$ is called \emph{even} if for any $x\in A$  $f(-x)=f(x)$.
\end{definition}
By definition, if $f\colon A\to B$ is a group homomorphism then $f$ is odd.

\begin{proposition}\label{prop:signed_index_maps}
Let $\tau$ be a signed index with coefficients in a group $A$ and $f\colon A\to B$ be a map to another group $B$. Then
\begin{itemize}
  \item if $f$ is odd then $f\circ\tau$ is a signed index with coefficients in $B$;
  \item if $f$ is even then $f\circ\tau$ is an index with coefficients in $B$.
\end{itemize}
If $\pi$ be an index with coefficients in $A$ and $f\colon A\to B$ be an arbitrary map then $f\circ\pi$ is an index with coefficients in $B$.
\end{proposition}

\begin{example}
Let $p$ be an oriented parity with coefficients in an abelian group $A$. Then its product with the sign function $sgn\cdot p$ is an index. If $\pi$ is an index with coefficients in $A$ then $sgn\cdot\pi$ is a signed index. For a signed index $\tau$, the product $sgn\cdot\tau$ is an index.

If $A=\Z$ (for example, $p=ip$ is the index parity) then the modulus $|p|$ ($|p|_D(v)=|p_D(v)|$, $v\in\V(D)$) of the parity $p$ is an index. The sign $sgn(p)$ of the parity $p$ where $sgn(p)_D(v)=sgn(p_D(v))$, is a signed index with values $-1,0,1$. The signed index $sgn(p)$ can be used as a substitution for the crossing sign in the case of flat knots.
\end{example}

\subsection{Linking invariant}

Let $\tau$ be a signed index on diagrams of a virtual knot $\K$ with coefficients in an abelian group $A$. Denote the value of $\tau$ on the loop crossings of type $l_+$ (see Fig.~\ref{fig:loop_types}) by $\tau^\circ$, and the value of crossings of type $r_+$ by $\tau^\bullet$. Then the crossings of type $l_-$ have the signed index $-\tau^\bullet$ and the crossings of type $r_-$ have the signed index $-\tau^\circ$.

\begin{definition}
The signed index $\tau$ is called \emph{R1-reduced} if $\tau^\circ=\tau^\bullet=0$.
\end{definition}

Note that any parity $p$ is R1-reduced.

\begin{proposition}
Let $\tau$ be an R1-reduced signed index on diagrams of a virtual knot $\K$ with coefficients in an abelian group $A$. Then the value
\begin{equation}\label{eq:signed_index_linking_invariant}
  lk(\tau)_D=\sum_{v\in\V(D)}\tau_D(v),\quad D\in\mathfrak K,
\end{equation}
is invariant under Reidemeister moves
\end{proposition}

\begin{proof}
Let $f\colon D\to D'$ be an increasing first Reidemeister move and $v'_0\in\V(D')$ be the new crossing. Then
\[
lk(\tau)_{D'}=\sum_{v'\in\V(D')}\tau_{D'}(v)=\sum_{v\in\V(D)}\tau_{D'}(f_*(v))+\tau_{D'}(v'_0)=\sum_{v\in\V(D)}\tau_{D}(v)=lk(\tau)_D
\]
where $\tau_{D'}(v'_0)=0$ because $\tau$ is R1-reduced.

Let $f\colon D\to D'$ be an increasing second Reidemeister move and $v'_1, v'_2\in\V(D')$ be the new crossings. Then
\begin{multline*}
lk(\tau)_{D'}=\sum_{v'\in\V(D')}\tau_{D'}(v)=\sum_{v\in\V(D)}\tau_{D'}(f_*(v))+\tau_{D'}(v'_1)+\tau_{D'}(v'_2)=\\
\sum_{v\in\V(D)}\tau_{D}(v)=lk(\tau)_D
\end{multline*}
because $\tau_{D'}(v'_1)=-\tau_{D'}(v'_2)$.

For a third Reidemeister move $f\colon D\to D'$ we have
\[
lk(\tau)_{D'}=\sum_{v'\in\V(D')}\tau_{D'}(v)=\sum_{v\in\V(D)}\tau_{D'}(f_*(v))=\sum_{v\in\V(D)}\tau_{D}(v)=lk(\tau)_D.
\]
\end{proof}

\begin{definition}
For an R1-reduced signed index $\tau$, the invariant $lk(\tau)=lk(\tau)_D\in A$ is called the \emph{linking invariant} of the signed index $\tau$.

It $\pi$ is an R1-reduced index on diagrams the knot $\K$ with coefficients in $A$, we define its \emph{linking invariant} as the linking invariant of its product with the sign function: $lk(\pi)=lk(sgn\cdot\pi)$. Then
\begin{equation}\label{eq:index_linking_invariant}
  lk(\pi)=\sum_{v\in\V(D)}sgn(v)\cdot\pi_D(v)
\end{equation}
for any diagram $D$ of the knot $\K$.
\end{definition}

\begin{example}
Let $D=D_1\cup D_2$ be a diagram of a two-component link $\mathcal L$. Consider the index $\pi$ with coefficients in $\Z$ given by the formula
\[
\pi_D(v)=\left\{\begin{array}{cl}
                            1, & v\mbox{ is a mixed crossing},\\
                            0, & v\mbox{ is a self-crossing}.
                          \end{array}\right.
\]
Then $\pi$ is an R1-reduced index, and its linking invariant
\[
  lk(\pi)=\sum_{v\in\V(D)}sgn(v)\cdot\pi(v)=\sum_{v\mbox{\scriptsize\ is a mixed crossing}}sgn(v)=lk(\mathcal L)
\]
is the linking number of the link $\mathcal L$.
\end{example}

\begin{remark}
Let $\tau$ be a signed index on diagrams of $\K$ with coefficients in $A$.
If $\tau$ is not R1-reduced define its \emph{R1-reduction} $\mathring\tau$ by the formula
\[
\mathring\tau_D(v)=\left\{\begin{array}{cl}
                            \tau_D(v), & \tau_D(v)\ne\pm\tau^\circ\mbox{ or }\pm\tau^\bullet, \\
                            0, & \tau_D(v)=\pm\tau^\circ\mbox{ or }\pm\tau^\bullet.
                          \end{array}\right.
\]
Then $\mathring\tau$ is an R1-reduced signed index with coefficients in $A$. We define the \emph{reduced linking invariant} by the formula
\begin{equation}\label{eq:signed_index_reduced_linking_invariant}
  \mathring{lk}(\tau)=\sum_{v\in\V(D)}\mathring\tau_D(v)=\sum_{v\in\V(D)\colon\tau_D(v)\ne\pm\tau^\circ,\pm\tau^\bullet}\tau_D(v).
\end{equation}

Another way to define a linking invariant for an unreduced signed index $\tau$ is to take the factor-group $\bar A=A/\left<\tau^\circ,\tau^\bullet\right>$ and consider the signed index $\bar\tau$ that is the composition of $\tau$ with the natural homomorphism $A\to\bar A$. By definition, the signed index $\bar\tau$ is R1-reduce, and we can define the linking invariant $\overline{lk}(\tau)=lk(\bar\tau)\in\bar A$.

For an index $\pi$ on diagrams of $\K$ with coefficients in $A$, we define the reduced linking invariants $\mathring{lk}(\pi)\in A$ and $\overline{lk}(\pi)\in\bar A=A/\left<\pi^\circ,\pi^\bullet\right>$  as the invariants of the signed index $sgn\cdot\pi$.
\end{remark}

\begin{example}
Let $\pi$ be an index on diagrams of a virtual knot $\K$ with coefficients in an abelian group $A$. Consider the map $G\colon A\to\Z[A]$, $G(x)=x$, $x\in A$. Then $G(\pi)$ be an index with coefficients in $\Z[A]$. The linking invariant $LK(\pi)=\mathring{lk}(G(\pi))\in\Z[A]$ is called the \emph{odd index polynomial} of the index $\pi$. For a signed index $\tau$, we define the odd index polynomial by  $LK(\tau)=LK(sgn\cdot\tau)$.

For example, if $\pi=W_K$ is the conventional $\Z$-valued index (see~\cite{K2}) then the linking invariant $LK(\pi)\in\Z[\Z]\simeq\Z[t]$ gives the known odd index (writhe~\cite{Cheng}, wriggle~\cite{FK}) polynomial.
\end{example}

\subsection{Inner product of parities}

The linking invariant $lk$ allows to define an invariant inner product on the space of parities.

\begin{definition}\label{def:parity_product}
Let $p_1$ be an oriented parity on diagrams of a virtual knot $\K$ with coefficients in an abelian group $A$, and $p_2$ be an oriented parity on diagrams of $\K$ with coefficients in an abelian group $B$. Define the \emph{inner product} of the parities $p_1$ and $p_2$ as the linking invariant of the index $p_1\otimes p_2$:
\begin{equation}\label{eq:parity_product}
  \left<p_1,p_2\right>=lk(p_1\otimes p_2)=\sum_{v\in\V(D)}sgn(v)\cdot(p_1)_D(v)\otimes(p_2)_D(v)\in A\otimes B.
\end{equation}
Analogously, one can define the inner product of any R1-reduced (signed) index with any (signed) index.
\end{definition}

\begin{remark}
1. The formula~\eqref{eq:parity_product} defines a bilinear map
\[
\mathcal{P}(\mathcal K,A)\times\mathcal{P}(\mathcal K,B)\to A\otimes B
\]
where $\mathcal{P}(\mathcal K,A)$ ($\mathcal{P}(\mathcal K,B)$)  is the abelian group of oriented parities on diagrams of $\K$ with coefficients in $A$ ($B$).

If $A$ is a ring (for example, $A=\Z$) then $\mathcal{P}(\mathcal K,A)$ becomes an $A$-module with the inner product $\left<\cdot,\cdot\right>$.

2. The signature $\sigma(\pi)$ of an index $\pi$ with coefficients in $A$ is equal to the inner product with the index parity $ip$:
\[
\sigma(\pi)=-\sum_{v\in\V(D)}\pi_D(v)\cdot ip_D(v)=-\left<\pi,ip\right>.
\]
\end{remark}

\subsection{Derived parities}

\begin{definition}
Let $p$ be oriented parity on diagrams of a virtual knot $\K$ with coefficients in an abelian group $A$. Then $sgn\cdot p$ is an R1-reduced index, and by Theorem~\ref{thm:index_parity_reduced} it defines the oriented parity $p'=p^{sgn\cdot p}$ with coefficients in the group $A'=A/\left<\sigma(sgn\cdot p)\right>$. We call the parity $p'$ the \emph{derived parity} to the parity $p$.
\end{definition}

We can iterate the definition and get the series of derivations $p',p'',\dots, p^{(n)},\dots$ of the parity $p$ with coefficients in $A',A'',\dots,A^{(n)},\dots$. Note that the coefficient groups are connected with epimorphisms $A\to A'\to A''\to\cdots$.

\begin{remark}
Fix a diagram $D$ of the knot $\mathcal K$. Then for $D$ the values an oriented parity $p$ on diagrams of $\K$ with coefficients in $A$ is determined by the vector $\mathbf p =(p_D(v))_{v\in\V(D)}\in A^{\V(D)}$.

Let $M=(D^l_v\cdot D^-_{v'})_{v,v'\in\V(D)}$ be the intersection matrix (with the elements in $\Z$). By formula~\eqref{eq:index_reduction_parity} the vector of values $\mathbf p'=(p'_D(v))_{v\in\V(D)}\in (A')^{\V(D)}$ of the derived parity is equal to
\[
\mathbf p'=M\mathbf p.
\]
For the higher derivation we have the formula $\mathbf p^{(n)}=M^n\mathbf p\in (A^{(n)})^{\V(D)}$.

The coefficients of the intersection matrix can be calculated with the formulas from~\cite{Turaev}.
\end{remark}

\begin{example}
Let $\K$ be a virtual knot and $p_0=ip$ be the index parity on $\K$ with coefficients in $A_0=\Z$. Denote the derivations of the index parity by $p_n=ip^{(n)}$, $n\ge 0$, and let $A_n$ be the coefficient group of $p_n$. Denote $\sigma_n=\sigma(p_n)\in A_n$ and $LK_n=LK(p_n)\in\Z[A_n]$. Note that $A_{n+1}=A_n/\left<\sigma_n\right>$.

Calculations of the derived parities for virtual knots with number of crossings $\le 4$ demonstrate several patterns.

1. (degeneration) Consider the knot 3.1 (Fig.~\ref{fig:virtual_knots}). The signature of derived parities are equal $\sigma_0=4$, $\sigma_1=0$ and $\sigma_2=1$. Then $A_n=0$ for all $n\ge 3$, and the parities $p_n$, $n\ge 3$, vanish. The odd index polynomials are
\[
LK_0=-t^{-1}+t-t^2,\quad LK_1=-t^{-1},\quad LK_n=0,\ n\ge 2.
\]

Another degeneration example is given by the knot 3.5. For this knot, the signatures are $\sigma_1=8$ and $\sigma_n=0$ for $n\ge 1$. Then $A_n=\Z_8$ for $n\ge 2$. On the other hand, $p_3=0$, so $p_n=0$ for all $n\ge n$. The odd index polynomials are
\[
LK_0=-t^{-2}-t^2,\quad LK_1=-3t^4,\quad LK_2=-2t^4,\quad LK_n=0,\ n\ge 3.
\]

\begin{figure}[h]
\centering\includegraphics[width=0.15\textwidth]{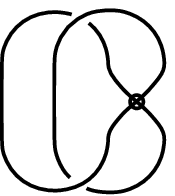}\quad\includegraphics[width=0.15\textwidth]{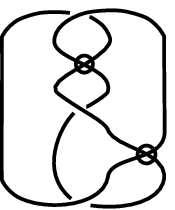}
\quad\includegraphics[width=0.15\textwidth]{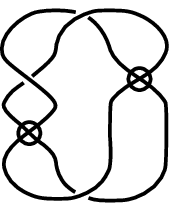}\quad\includegraphics[width=0.15\textwidth]{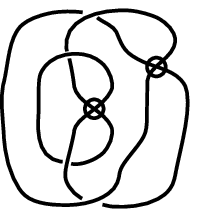}
\quad\includegraphics[width=0.15\textwidth]{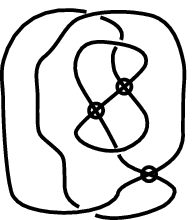}
\caption{Virtual knots 2.1, 3.1, 3.5, 4.1, 4.75}\label{fig:virtual_knots}
\end{figure}

2. (stabilization) Consider the knot 2.1. Then $\sigma_0=2$ and $\sigma_n=0$ for $n\ge 1$. Then $A_n=\Z_2$ for $n\ge 1$. For $n\ge 1$ the derived parities are equal to $p_n=(1,1)$. The odd index polynomials are
\[
LK_0=-t^{-1}-t,\quad LK_n=-2t^{-1},\ n\ge 1.
\]

3. (periodicity) Consider the knot 4.1. For this knot, $\sigma_0=4$ and $\sigma_n=0$ for $n\ge 1$. Then $A_n=\Z_4$, $n\ge 1$. The sequence of derived parities has period $4$: $\mathbf p_{4k}=(1,-1,1,-1)$, $\mathbf p_{4k+1}=(1,1,1,1)$, $\mathbf p_{4k+2}=(-1,1,-1,1)$, $\mathbf p_{4k+3}=(-1,-1,-1,-1)$, $k\ge 0$. (The numeration of the crossings coincides with that in the Gauss code in the Green's table of virtual knots~\cite{Green}.) The odd index polynomials are
\[
LK_{4k}=LK_{4k+2}=-2t^{-1}-2t,\quad LK_{4k+1}=-4t^{-1},\quad LK_{4k+3}=-4t,\ k\ge 0.
\]

4. (growth) Consider the knot 4.75. Then $\sigma_n=0$ for all $n$, hence, $A_n=\Z$, $n\ge 0$. The parity vector $\mathbf p_0$ appears to be the eigenvector of the intersection matrix $M$ with the eigenvalue $2$. Thus, $\mathbf p_n=2^n\mathbf p_0=(2^n,-2^n,-2^n,2^n)$. The odd index polynomials are all zero: $LK_n=0$, $n\ge 0$.

The knot 4.107 gives an example of growth with periodicity. For this knot, $A_n=\Z$ for all $n$, and $\mathbf p_{2k}=(2\cdot 3^k,-2\cdot 3^k,-2\cdot 3^k,2\cdot 3^k)$ and $\mathbf p_{2k+1}=(-2\cdot 3^k,-2\cdot 3^k,-2\cdot 3^k,-2\cdot 3^k)$, $k\ge 0$. The odd index polynomials are
\[
LK_{2k}=0,\quad LK_{2k+1}=2t^{-2\cdot 3^k}-2t^{2\cdot 3^k},\ k\ge 0.
\]
\end{example}

\begin{example}
Consider the knots 2.1 and 4.4 from the table~\cite{Green}. We denote them by $K_1$ and $K_2$. Then $LK_0(K_1)=LK_0(K_2)=-t^{-1}-t$ but $LK_1(K_1)=-2t^{-1}$ and $LK_1(K_2)=0$ are different. Thus, the derived index polynomials can be more sensitive than the conventional one.

On the other hand, if index parity is trivial ($ip\equiv 0$) then the derived parities are also trivial. This situation occurs for almost classical knots.
\end{example}

\section{Open questions}

We conclude the paper with some open questions.

\begin{enumerate}
  \item In Section~\ref{subsect:cycle_biquandle} it was shown that cohomology of biquandles generates parity cycles and parities. Is it true that any parity can be obtained by this construction? In particular, for an knot $K$ describe the cohomology group $H^1(B(K),A)$ of the fundamental biquandle $B(K)$ of $K$.
  \item Describe the colouring monodromy groups $Mon_B(\K)$. Which knots have trivial monodromy? Describe the (quasi)index monodromy groups $Mon(\pi)$. Which indices $\pi$ have trivial monodromy?
  \item Find nontrivial examples of quasi-indices which are not indices. What is the meaning of the conditions~\eqref{eq:index_conditions} on biquandle 1-cocycles?
\end{enumerate}

\end{document}